\documentclass[12pt]{amsart}
\usepackage{amsfonts, amssymb, euscript, mathrsfs}
\usepackage{upref}
\usepackage[usenames, dvipsnames]{color}
\usepackage{url,graphicx,tabularx,array}
\usepackage{array} % Improves `tabular` and `array` environments
\usepackage{pict2e} % Allows \linethickness{...} in diagonal lines
\usepackage{subfigure}
\usepackage{booktabs}
\newcommand{\ra}[1]{\renewcommand{\arraystretch}{#1}}

\usepackage[english]{babel}
\usepackage{fancyhdr}
\usepackage{tikz}
\usepackage{enumerate}
\usepackage[margin=0.8in]{geometry}
\usepackage{aliascnt}
\usepackage{hyperref}
\usepackage{verbatim}

\usepackage{amssymb}
\usepackage{amsfonts}
\usepackage{amsmath}
\usepackage[singlespacing]{setspace}
\usepackage{indentfirst}
\usepackage{endnotes}

\usepackage{float}
\usepackage{multirow}


\allowdisplaybreaks[4]

\newcommand{\BR}{\mathbb{R}}

\newcommand{\ME}{\mathbb E}

\newcommand{\MP}{\mathbb P}

\newcommand{\CA}{\mathcal A}
\newcommand{\CL}{\mathcal L}

\newcommand{\si}{\sigma}

\renewcommand{\phi}{\varphi}

\newcommand{\eps}{\varepsilon}
\newcommand{\la}{\lambda}

\newcommand{\Ra}{\Rightarrow}

\newcommand{\ol}{\overline}

\newcommand{\norm}[1]{\lVert#1\rVert}
\renewcommand{\comment}[1]{}

\newcommand{\md}{\mathrm{d}}

\DeclareMathOperator{\Var}{Var}

\DeclareMathOperator{\Exp}{Exp}

\theoremstyle{plain}
\newtheorem{thm}{Theorem}[section]

\newtheorem{lemma}[thm]{Lemma}

\newtheorem{corollary}[thm]{Corollary}

\numberwithin{equation}{section}

\theoremstyle{definition}
\newtheorem{definition}{Definition}[section]
\newtheorem{assumption}{Assumption}[section]

\theoremstyle{remark}
\newtheorem{remark}{Remark}[section]
\newtheorem{ex}{Example}

\begin{document}

\title[Convergence rate of ruin probabilities]{Exponential convergence rate of ruin probabilities\\ for level-dependent L\'evy-driven risk processes}

\author{Pierre-Olivier Goffard}

\address{\scriptsize Department of Statistics and Applied Probability, University of California, Santa Barbara}

\email{goffard@pstat.ucsb.edu}

\author{Andrey Sarantsev}

\address{\scriptsize Department of Statistics and Applied Probability, University of California, Santa Barbara}

\email{sarantsev@pstat.ucsb.edu}

\begin{abstract} We explicitly find the rate of exponential long-term convergence for the ruin probability in a level-dependent L\'evy-driven risk model, as time goes to infinity. Siegmund duality allows to reduce the pro blem to long-term convergence of a reflected jump-diffusion to its stationary distribution, which is handled via Lyapunov functions.
\end{abstract}


%\ams{91B30}{60H10, 60J60, 60J75}


\maketitle

\thispagestyle{empty}

\section{Introduction}
A non-life insurance company holds at time $t=0$ an initial capital $u = X(0)\geq0$, collects  premiums at a rate $p(x)>0$ depending on the current level of the capital $X(t)=x$, and pays from time to time a compensation (when a claim is filed). The aggregated size of claims up to time $t>0$ is modeled by a compound Poisson process $(L(t)\text{ , }t\geq0)$. That is, the number of claims is governed by a homogeneous Poisson process of intensity $\beta$ independent from the claim sizes. The claim sizes, in turn, form a sequence $U_1,U_2,\ldots$ of i.i.d. nonnegative random variables with cumulative distribution function $B(\cdot)$. The net worth of the insurance company is then given by a continuous-time stochastic process $X = (X(t),\, t \ge 0)$, with
\begin{equation}\label{eq:CP-classic}
X(t) = u+\int_0^{t}p(X(s))\text{d}s-\sum_{k=1}^{N(t)}U_k=u+\int_0^{t}p(X(s))\,\text{d}s-L(t),\ t \ge 0.
\end{equation}
Examples of such level-dependent premium rate include the insurance company  downgrading the premium rate from $p_1$ to $p_2$ when the reserves reach a certain threshold; or incorporating a constant interest force: $p(x)=p+ix$. In this work, a more general risk model is considered. The surplus \eqref{eq:CP-classic} is perturbed by a Brownian motion $\{W(t)\text{ , }t\geq0\}$, multiplied by a diffusion parameter $\si$, to account for the fluctuations around the premium rate. This diffusion parameter may also depend on $X(t)$. We further let the accumulated liability $L(t)$ be governed by a pure jump nondecreasing L\'evy process, starting from $L(0) = 0$. The financial reserves of the insurance company evolve according to the following dynamics:
\begin{equation}
\label{eq:LevelDependentLevyDrivenRiskProcess}
\text{d}X(t)= p(X(t))\,\text{d}t+\sigma(X(t))\,\text{d}W(t)-\text{d}L(t),\quad X(0) = u.
\end{equation}
In risk theory, one of the main challenges is the evaluation of ruin probabilities. The probability of ultimate ruin is the probability that the reserves ever drop below zero:
\begin{equation}
\label{eq:ruin-infinite-horizon}
\psi(u) = \MP\bigl(\inf\limits_{t \ge 0}X(t) \le 0\bigr).
\end{equation}
We stress dependence of $\psi$ on the initial capital $u$. The probability of ruin by time $T$ is defined as
\begin{equation}
\label{eq:ruin-finite-horizon}
\psi(u, T) := \MP\bigl(\inf\limits_{0 \le t \le T}X(t) \le 0\bigr).
\end{equation}
We often refer to $\psi(u)$ and $\psi(u,T)$ as ruin probabilities for infinite and finite time horizon, respectively. For a comprehensive overview on risk theory and ruin probabilities, see the book~\cite{AsAl10}.

We study the rate of exponential convergence of the finite-time horizon ruin probability toward its infinite-time counterpart. The goal of this article is to provide an explicit estimate for such rate: To find constants $C, k > 0$ such that
\begin{equation}
\label{eq:exp-convergence}
0 \le \psi(u) - \psi(u, T) \le Ce^{-kT},\ \ \mbox{for all}\ \ T, u \ge 0.
\end{equation}
This is achieved via a duality argument. For the original model~\eqref{eq:CP-classic}, define the {\it storage process} $Y=\{Y(t)\text{ , }t\geq0\}$ as follows:
\begin{equation}\label{eq:StorageProcess}
Y(t)=L(t)-\int_{0}^{t}p(Y(s))\,\md s.
\end{equation}
We assume that $p(y) = 0$ for $y < 0$. This makes zero a reflecting barrier. This is essentially a time-reversed version of the risk model \eqref{eq:CP-classic}, reflected at $0$. For the general model~\eqref{eq:LevelDependentLevyDrivenRiskProcess} perturbed by Brownian motion, the dual process is a reflected jump-diffusion on the positive half-line. As $t \to \infty$, $Y(t)$ weakly converges to some distribution $Y(\infty)$. The crucial observation is: For $T > 0$ and $u \ge 0$,
$$
\MP(Y(T) \ge u) = \psi(u, T),\quad \MP(Y(\infty) \ge u) = \psi(u).
$$
This is a particular case of Siegmund duality, see Siegmund \cite{Siegmund}. This method was first employed in \cite{Levy}, for the similar duality between absorbed and reflected Brownian motion. It has become a standard tool in risk theory since the seminal paper of Prabhu \cite{Pa61}, see also \cite[Chapter III, Section 2]{AsAl10}. The problem ~\eqref{eq:exp-convergence} therefore reduces to the study of the convergence of $Y(t)$ toward $Y(\infty)$ as $t \to \infty$:
$$
0 \le \mathbb{P}(Y(\infty)>u) - \MP(Y(T) \ge u) \le Ce^{-kT}.
$$
A \textit{stochastically ordered} real-valued Markov process $Y = \{Y(t)\text{ , }t\geq0\}$ is such that, for all $y_1 \ge y_2$, we can couple two copies $Y_1(t)$ and $Y_2(t)$ of $Y(t)$ starting from $Y_1(0) = y_1$ and $Y_2(0) = y_2$, in such a way that $Y_1(t) \ge Y_2(t)$ a.s. for all $t \ge 0$. A {\it Lyapunov function} for a Markov process with generator $\mathcal L$ is, roughly speaking, a function $V \ge 1$ such that  $\mathcal L V(x) \le -cV(x)$ for some constant $c > 0$, for all $x$ outside of a compact set.  Then we can combine this coupling method with a Lyapunov function to get a simple, explicit, and in some cases, sharp estimate for the rate $k$. This method was first applied in Lund and Tweedie \cite{LT1996} for discrete-time Markov chains, and in Lund et al. \cite{LMT1996} for continuous-time Markov processes. A direct application of their results yields the rate of convergence for the storage process defined in \eqref{eq:StorageProcess} and the level-dependent compound Poisson risk model \eqref{eq:CP-classic}. However, the dual model associated to the risk process~\eqref{eq:LevelDependentLevyDrivenRiskProcess} is a more general process: This is a reflected jump-diffusion on the positive half-line.

\smallskip

The same method as in Lund et al. \cite{LMT1996} has been refined in a recent paper by Sarantsev \cite{MyOwn12} and applied to reflected jump-diffusions on the half line. The jump part is not a general L\'evy process, but rather a state-dependent compound Poisson process, which makes a.s. finitely many jumps in finite time. In a recent paper \cite{MyOwn16}, it was applied to {\it Walsh diffusions} (processes which move along the rays emanating from the origin in $\BR^d$ as one-dimensional diffusions; as they hit the origin, they choose a new ray randomly). Without attempting to give an exhaustive survey, let us mention classic papers \cite{DMT1995, MT1993a, MT1993b} which use Lyapunov functions (without stochastic ordering) to prove the very fact of exponential long-term convergence, and a related paper of Sarantsev \cite{MyOwn10}. However, to estimate the rate $k$ explicitly is a harder problem. Some partial results in this direction are provided in the papers~\cite{BCG2008, Davies, Explicit, RR1996, RT1999, RT2000}.
%In the latter paper, the rate $k$ for a reflected Brownian motion with negative drift on the positive half-line.
\smallskip

In this paper, we combine these two methods: Lyapunov functions and stochastic ordering, to find the rate of convergence of the process $Y$, which is dual to the original process $X$ from~\eqref{eq:LevelDependentLevyDrivenRiskProcess}. This process $Y$, as noted above, is a reflected jump-diffusion on the half-line. We apply the same method developed in \cite{LMT1996, MyOwn12}. In the general case, it can have infinitely many jumps during finite time, or can have no diffusion component, as in the level dependent compound Poisson risk model from~\eqref{eq:CP-classic}. Therefore, we need to adjust the argument from \cite{MyOwn12}. Our method only applies in the case of light tailed claim size. Asmussen and Teugels in \cite{AsTe96} studied the convergence of ruin probabilities in the compound Poisson risk model with sub-exponentially distributed claim size. It is shown that the convergence takes place at a sub-exponential rate.

\smallskip

The paper is organized as follows. In Section 2, we define assumptions on $p$, $\sigma$, and the L\'evy process $L$. We also introduce the concept of Siegmund duality to reduce the problem to convergence rate of a reflected jump-diffusion to its stationary distribution. Our main results are stated in Section 3: Theorem~\ref{thm:main-1} and Corollary \ref{cor:main-1} provide an estimate for the exponential rate of convergence.  Section \ref{sec:ComputationConvergenceRate} gives examples of calculations of the rate $k$. The proof of Theorem~\ref{thm:main-1} is carried out in Section 5. Proofs of some technical  lemmata are postponed until Appendix.

\section{Definitions and Siegmund duality}
First, let us impose assumptions on our model~\eqref{eq:LevelDependentLevyDrivenRiskProcess}. Recall that the wealth of the insurance company is modeled by the right-continuous process with left limits $X = (X(t),\, t \ge 0)$, governed by the following integral equation:
$$
X(t) = u + \int_0^tp(X(s))\,\md s + \int_0^t\si(X(s))\,\md W(s) - L(t),
$$
or, equivalently, by the stochastic differential equation (SDE) with initial condition $X(0) = u$, given by \eqref{eq:LevelDependentLevyDrivenRiskProcess}. We say that $X$ is {\it driven} by the Brownian motion $W$ and L\'evy process $L$.
%\begin{equation}
%\label{eq:LevelDependentLevyDrivenRiskProcess}
%\md X(t) = p(X(t))\,\md t + \si(X(t))\,\md W(t) - \md L(t).
%\end{equation}
A function $f : \BR \to \BR$, or $f : \BR_+ \to \BR$, is {\it Lipschitz continuous} if there exists a constant $K$ such that $|f(x) - f(y)| \le K|x-y|$ for all $x$ and $y$.
\begin{assumption}\label{as:1}
The function $p : \BR_+ \to \BR$ is Lipschitz. The function $\si : \BR_+ \to \BR_{\begin{color}{blue}+\end{color}}$ is bounded, and continuously differentiable with Lipschitz continuous derivative $\si'$.
\label{asmp:Lipschitz}
\end{assumption}
\begin{assumption}
The process $L$ is a pure jump subordinator, that is, a L\'evy process (stationary independent increments) with $L(0) = 0$, and with a.s. nondecreasing trajectories, which are right continuous with left limits. The process $W$ is a standard Brownian motion, independent of $L$.
\label{asmp:Levy}
\end{assumption}
Assumption \ref{asmp:Lipschitz} is not too restrictive as it allows to consider classical risk process such as: (a) the compound Poisson risk process when $p(x)=p$, and $\sigma(x)=0$; (b) the compound Poisson risk process under constant interest force when $p(x)=p+ix$, and $\sigma(x)=0$. However, the regime-switching premium rate when the surplus hits some target is not covered.

\smallskip

Assumption \ref{asmp:Levy} allows the study of the compound Poisson risk process perturbed by a diffusion when $p(x)=p$, and $\sigma(x)=\sigma$, extensively discussed in the paper by Dufresne and Gerber \cite{DuGe91}, as well as the L\'evy-driven risk process defined for example in Morales and Schoutens \cite{MoSc03}. It is known from the standard theory, see for example \cite[Section 6.2]{KS1998}, that the {\it L\'evy measure} of this process is a measure $\mu$ on $\BR_+$ which satisfies
\begin{equation}
\label{eq:classic-mean}
\int_0^{\infty}(1\wedge x)\,\mu(\md x) < \infty.
\end{equation}
From Assumption~\ref{asmp:Levy}, we have:
$$
\ME e^{-\lambda L(t)} = \exp\left(t\kappa(-\lambda)\right),\ \ \mbox{for every}\ \ t, \lambda \ge 0,
$$
where $\kappa(\lambda)$ is the {\it L\'evy exponent:}
\begin{equation}
\label{eq:levy-exp}
\kappa(\lambda) := \int_0^{\infty}\left[e^{\lambda x} - 1\right]\mu(\md x),\ \lambda \in \BR.
\end{equation}

Under Assumptions~\ref{asmp:Lipschitz} and~\ref{asmp:Levy}, $L$ is a Feller continuous strong Markov process, with generator
\begin{equation}
\label{eq:explicit-generator}
\mathcal N f(x) = \int_0^{\infty}\left[f(x+y) - f(x)\right]\,\mu(\md y),
\end{equation}
for $f \in C^2(\BR)$ with compact support. For our purposes, we impose an additional assumption.

\begin{assumption} The  measure $\mu$ has finite exponential moment: for some $\lambda_0 > 0$, we have
\begin{equation}
\label{eq:finite-exp}
\int_1^{\infty}e^{\lambda_0 x}\,\mu(\md x) < \infty.
\end{equation}
\label{asmp:finite-exp}
\end{assumption}
\begin{remark}
The existence of exponential moments on the jump sizes distribution prevent us from considering heavy tailed claim size distribution as in Asmussen and Teugels \cite{AsTe96}.
\end{remark}
Under Assumption~\ref{asmp:finite-exp}, we can combine~\eqref{eq:classic-mean} and~\eqref{eq:finite-exp} to get:
$$
\kappa(\lambda) < \infty \ \ \mbox{for}\ \ \lambda \in [0, \lambda_0).
$$
Then we can extend the formula~\eqref{eq:explicit-generator} for functions $f \in C^2(\BR)$ which satisfy
\begin{equation}
\label{eq:exp-bdd}
\sup\limits_{x \ge 0}e^{-\lambda x}|f(x)| < \infty\ \mbox{for some}\ \la \in (0, \la_0).
\end{equation}

The proof of the following technical lemma is postponed to the Appendix \ref{appendix:ProofLemma1}.
\begin{lemma} Under Assumptions~\ref{asmp:Levy} and~\ref{asmp:finite-exp}, the following quantity is finite:
\label{lemma:tech}
\begin{equation}
\label{eq:Mean}
m(\mu) := \int_0^{\infty}x\,\mu(\md x) < \infty.
\end{equation}
\end{lemma}

\begin{ex}
If $\{L(t),\, t \ge 0\}$ is a compound Poisson process with jump intensity $\beta$ and distribution $B$ for each jump, then the L\'evy measure is given by $\mu(\cdot) = \beta B(\cdot)$.
\end{ex}

The following lemma can be proved by a classic argument, a version of which can be found in any textbook on stochastic analysis, see for example \cite[Section 5.2]{KS1998}. For the sake of completeness, we give the proof in the Appendix \ref{appendix:ProofLemma2}.

\begin{lemma}\label{lem:ExistenceXProcess} Under Assumptions~\ref{asmp:Lipschitz} and~\ref{asmp:Levy}, for every initial condition $X(0) = u$ there exists (in the strong sense, that is, on a given probability space) a pathwise unique version of~\eqref{eq:LevelDependentLevyDrivenRiskProcess}, driven by the given Brownian motion $W$ and L\'evy process $L$. This is a Markov process, with generator
\begin{equation}
\label{eq:generator-absorbed}
\CL f(x) := p(x)f'(x) + \frac12\si^2(x)f''(x) + \int_0^{\infty}[f(x-y) - f(x)]\,\mu(\md y)
\end{equation}
for $f \in C^2(\BR)$ with compact support.  Under Assumption~\ref{asmp:finite-exp}, this expression~\eqref{eq:generator-absorbed} is also valid for functions $f \in C^2(\BR)$ satisfying~\eqref{eq:exp-bdd} with $f(-x)$ instead of $f(x)$.
\end{lemma}

%\begin{proof} For Lipschitz continuous $p$, this In this case, we actually have strong existence and pathwise uniqueness. For general locally bounded $p$, this follows from uniqueness for $p = 0$ and removal of drift using the Girsanov theorem. We also need to prove that the process is {\it conservative:} That is, explosion time is a.s. equal to infinity. This can be done by estimating $\ME\sup_{0 \le s \le t}X^2(s)$ using martingale inequalities and sublinear growth of $p$. We omit details of the proof, since they are very similar to standard classical arguments.
%\end{proof}

%\begin{asmp} The process $X$ is {\it conservative:} It does not explode, and is well-defined for all $t \ge 0$.
%\label{asmp:conservative}
%\end{asmp}
%
%Assumption~\ref{asmp:conservative} holds, for example, when the drift $p$ is uniformly bounded, or uniformly Lipschitz.

%\smallskip

Define the {\it ruin probability} in finite and infinite time horizons as in~\eqref{eq:ruin-finite-horizon} and~\eqref{eq:ruin-infinite-horizon}. We are interested in finding an estimate of the form
$$
0 \le \psi(u) - \psi(u, T) \le Ce^{-kT},\ u, T \ge 0,
$$
for some constants $C,\,k > 0$. Recall the concept of {\it Siegmund duality}.
\begin{definition} Two Markov processes $X = (X(t),\, t \ge 0)$ and $Y = (Y(t),\, t \ge 0)$ on $\BR_+$ are called {\it Siegmund dual} if for all $t, x, y \ge 0$,
$$
\MP_x(X(t) \ge y) = \MP_y(Y(t) \le x).
$$
Here, the indices $x$ and $y$ refer to initial conditions $X(0) = x$ and $Y(0) = y$.
\end{definition}
Using Siegmund duality allow us to reduce our problem about ruin probabilities to another problem: long-term convergence to the stationary distribution of a reflected jump-diffusion $Y=\{Y(t)\text{ , }t\geq0\}$. Take some functions $p_{\ast},\, \sigma_{\ast} :\mathbb{R}_+\to\mathbb{R}$.

\begin{definition} Consider an $\BR_+$-valued process $Y = (Y(t),\, t \ge 0)$ with right-continuous trajectories with left limits, which satisfies the following SDE:
\begin{equation}
\label{eq:RSDE}
Y(t) = Y(0) + \int_0^tp_{\ast}(Y(s))\,\md s + \int_0^t\sigma_{\ast}(Y(s))\,\md W(s) + L(t) + R(t),
\end{equation}
where $R = (R(t),\, t \ge 0)$ is a nondecreasing right-continuous process with left limits, which starts from $R(0) = 0$ and can increase only when $Y(t) = 0$. Then the process $Y$ is called a {\it reflected jump-diffusion on the half-line}, with {\it drift coefficient} $p_{\ast}$, {\it diffusion coefficient} $\si_{\ast}$, and {\it driving jump process} $L$ with {\it L\'evy measure} $\mu$.
\end{definition}
The following result is the counterpart of Lemma \ref{lem:ExistenceXProcess} for the process $Y=\{Y(t)\text{ , }t\geq0\}$.
\begin{lemma}\label{lemma:ExistenceRProcess}
If $p_{\ast}$ and  $\si_{\ast}$ are Lipschitz, then for every initial condition $Y(0) = y$, there exists in the strong sense a pathwise unique version of~\eqref{eq:RSDE}. This is a Markov process with generator $\CA$, given by the formula
\begin{equation}
\label{eq:A}
\CA f(x) = p_{\ast}(x)f'(x) + \frac12\si_{\ast}^{2}(x)f''(x) + \int_0^{\infty}\left[f(x + y) - f(x)\right]\,\mu(\md y),
\end{equation}
for $f \in C^2(\BR_+)$ with compact support and with $f'(0) = 0$.
\label{lem:ExistenceRProcess}
\end{lemma}

The proof, which is similar to that of Lemma~\ref{lem:ExistenceXProcess}, is provided in the Appendix \ref{appendix:ProofLemma3}.

\smallskip

It was shown in \cite{Siegmund} that a Markov process on $\BR_+$ has a (Siegmund) dual process if and only if it is stochastically ordered.

\begin{thm} A Markov process $X$, corresponding to a transition semigroup $(P^t)_{t\geq0}$, is {\it stochastically ordered}, if and only if one of the following two conditions holds:

\smallskip

(a) the semigroup $(P^t)_{t \ge 0}$ maps bounded nondecreasing functions into bounded nondecreasing functions; that is, for every bounded nondecreasing $f : \BR_+ \to \BR$ and every $t \ge 0$, the function $P^tf$ is also bounded and nondecreasing;

\smallskip

(b) for every $t \ge 0$ and $c \ge 0$, the function $x \mapsto \MP_x(X(t) \ge c)$ is nondecreasing in $x$;

\smallskip
\label{thm:stoch-order}
% (c) for every $x_1 \ge x_2 \ge 0$, we can couple two copies $X_1$ and $X_2$ of the process $X$, starting from $X_1(0) = x_1$ and $X_2(0) = x_2$, so that a.s. for all $t \ge 0$, we have $X_1(t) \ge X_2(t)$.
\end{thm}
\begin{proof}
This equivalence follows from \cite{Kamae}.
\end{proof}
Now, consider the process~\eqref{eq:LevelDependentLevyDrivenRiskProcess}, stopped at hitting $0$. The following result is well known in the literature; however, in the Appendix~\ref{appendix:ProofStochOrder} we provide a simple proof for the sake of completeness.

\begin{lemma} The process~\eqref{eq:LevelDependentLevyDrivenRiskProcess} is stochastically ordered.
\label{lemma:stoch-ordered-original}
\end{lemma}

It was first shown in \cite[p.210]{Levy} that absorbed and reflected Brownian motions on $\BR_+$ are Siegmund dual. Since then, several more papers dealt with duality for more general processes, including jump-diffusions in \cite{Kol}. In particular, we have the following result.
\begin{lemma} Under Assumptions~\ref{asmp:Lipschitz} and~\ref{asmp:Levy}, the Siegmund dual process for the jump-diffusion~\eqref{eq:LevelDependentLevyDrivenRiskProcess}, absorbed at zero, is the reflected jump-diffusion on $\BR_+$ from~\eqref{eq:RSDE}, starting at $Y(0)=0$, with drift  and diffusion coefficients
\begin{equation}
\label{eq:new-drift}
p_{\ast}(x) = -p(x) - \si(x)\si'(x),
\end{equation}
\begin{equation}
\label{eq:new-difusion}
\si_{\ast}(x) = \si(x).
\end{equation}
\end{lemma}
\begin{proof}
The result is a direct application of \cite[Proposition 3.1]{Kol}
\end{proof}

We have shown that under Assumptions \ref{asmp:Lipschitz}, \ref{asmp:Levy} and \ref{asmp:finite-exp}, the wealth process is a stochastically ordered Markov process that admits as a Siegmund dual process a Markov process defined as a reflected jump-diffusion process. Therefore, the rate of convergence for ruin probabilities is determined by studying the one of its associated dual process $Y=\{Y(t)\text{ , }t\geq0\}$.

\section{Main results}

%which is a Markov process with generator
%\begin{equation}
%\CA f(x) = p_{\ast}(x)f'(x) + \frac12\si^2(x)f''(x) + \int_0^{\infty}\left[f(x + y) - f(x)\right]\,\mu(\md y),
%\end{equation}
%where
%\begin{equation}
%\label{eq:new-drift}
%p_{\ast}(x) = -p(x) + \si(x)\si'(x),
%\end{equation}

A common method to prove an exponential rate of convergence toward the stationary distribution is to construct a Lyapunov function.
\begin{definition}
Let $V:\mathbb{R}_+\to[1,\infty)$ be a continuous function and assume there exists $b,k,z>0$ such that
\begin{equation}\label{eq:LyapunovFunction}
\CA V(x)\leq-kV(x)+b{1}_{[0,z]}(x),\text{ }x\in\mathbb{R}_+.
\end{equation}
then $V$ is called a \textit{Lyapunov function}.
\end{definition}
We shall build a Lyapunov function for the Markov process $Y$ in the form $V_{\lambda}(x)=e^{\lambda x}$, for $\lambda>0$. This choice appears to be suitable to tackle the rate of convergence problem of reflected jump-diffusions process as the generator acts on it in a simple way. Under Assumption~\ref{asmp:finite-exp}, consider the function
$$
\phi(\la, x) := p_*(x)\la + \frac12\si^2(x)\la^2 + \kappa(\la),\ \la \in[0,\lambda_0),\ x \in \BR.
$$

For a signed measure $\nu$ on $\mathbb{R}_+$ and a function $f:\mathbb{R}_+\to\mathbb{R}$, we denote by $(\nu, f)=\int f\text{d}\nu$. Additionally, for a function $f:\mathbb{R}_+\to\left[1,+\infty\right)$, define the following norm: $\norm{\nu}_f:=\sup_{|g| \le f}|(\nu,g)|$. If $f\equiv1$, then $\norm{\cdot}_f$ is the total variation norm.  Define
\begin{equation}
\label{eq:Phi-definition}
\Phi(\lambda)=\inf\limits_{x \ge 0}(-\phi(\la, x)) = -\sup\limits_{x \ge 0}\phi(\la, x).
\end{equation}

\begin{thm}
\label{thm:main-1}
 Under Assumptions~\ref{asmp:Lipschitz},~\ref{asmp:Levy},~\ref{asmp:finite-exp}, suppose
\begin{equation}
\label{eq:ergodic}
\Phi(\lambda) > 0\ \mbox{for some}\ \la \in (0, \la_0).
\end{equation}
Then there exists a unique stationary distribution $\pi$ for the reflected jump-diffusion $Y$. Take a $\la \in (0, \la_0)$ such that $k = \Phi(\lambda) > 0$. This stationary distribution satisfies $(\pi, V_{\la})< \infty$. The transition function $Q^t(x, \cdot)$ of the process $Y$ satisfies
\begin{equation}
\label{eq:uniform-ergodicity}
\norm{Q^t(x, \cdot) - \pi(\cdot)}_{V_\lambda} \le \left[V_{\lambda}(x) + (\pi,V_{\lambda})\right]e^{-kt}.
\end{equation}
\end{thm}

The proof of Theorem~\ref{thm:main-1} is postponed until Section~\ref{sec:ProofMainThm}. The central result of this paper is a corollary of Theorem \ref{thm:main-1}, direct consequence of the duality link established between the processes $X$ and $Y$.
\begin{corollary} Under Assumptions~\ref{asmp:Lipschitz},~\ref{asmp:Levy}, ~\ref{asmp:finite-exp}, and the condition~\eqref{eq:ergodic},
\begin{equation}\label{eq:CorConvergenceRateRuinProbabilities}
0 \le \psi(u) - \psi(u, T) \le \left[1 + (\pi, V_{\la})\right]e^{-kT},\quad u, T \ge 0.
\end{equation}
\label{cor:main-1}
\end{corollary}

\begin{proof}
In virtue of Siegmund duality we have that
\begin{equation}\label{eq:DualityInCor}
 \psi(u) - \psi(u, T)=\MP(Y(\infty) \ge u) - \MP(Y(T) \ge u),
\end{equation}
where $Y=\left(Y(t)\text{ , }t\geq0\right)$ is a reflected jump-diffusion on $\mathbb{R}_{+}$, starting at $Y(0)=0$, and $Y(\infty)$ is a random variable distributed as $\pi$. We may rewrite \eqref{eq:DualityInCor} as
\begin{equation*}
 \psi(u) - \psi(u, T)=\pi\left(\left[u,\infty\right)\right) - Q^{T}(0,\left[u,\infty\right)).
\end{equation*}
Then the inequality \eqref{eq:CorConvergenceRateRuinProbabilities} follows immediately from the application of Theorem \ref{thm:main-1}.
\end{proof}

In the space-homogeneous case: $p(x) \equiv p$ and $\si(x) \equiv \si$, the quantity
 $\phi(\la, x)$ is independent of $x$, and condition~\eqref{eq:ergodic} means that there exists a $\la > 0$ such that $\phi(\la) < 0$. Then $p_{\ast} = p$, and
$$
\phi'(0) = -p + \psi'(0) = -p + m(\mu).
$$
It is easy to show that $\phi(\cdot)$ is a convex function with $\phi(0) = 0$. Therefore, condition~\eqref{eq:ergodic} holds if and only if $\phi'(0) < 0$, or, equivalently,
\begin{equation}
\label{eq:ergodic-hom}
p > m(\mu).
\end{equation}

\section{Explicit rate of exponential convergence calculation}\label{sec:ComputationConvergenceRate}
In this section, we aim at studying the rate $k$ of exponential convergence depending on the parameters of the risk model.
\subsection{Compound Poisson risk model perturbed by a diffusion}
In this subsection, the risk process $X=(X(t)\text{ , }t\geq0)$ is defined as
\begin{equation}
X(t)=u+pt+\sigma W(t)-\sum_{k=1}^{N(t)}U_k,
\end{equation}
where $u\geq0$ denotes the initial capital and $p$ corresponds to the premium rate. The process $W=(W(t)\text{ , }t\geq0)$ is a standard Brownian motion allowing to capture the volatility around the premium rate encapsulated in the parameter $\sigma>0$. The process $N=(N(t)\text{ , }t\geq0)$ is a homogeneous Poisson process with intensity $\beta>0$, independent from the claim sizes $U_1,U_2,\ldots$ which are \textbf{i.i.d.}  with distribution function $B$.
The premium rate satisfies the {\it net benefit condition:} $p=(1+\eta)\beta\mathbb{E}(U)$, where $\eta>0$ is {\it safety loading.}

\smallskip

We can study the rate of exponential convergence of ruin probabilities; specifically, how it depends on the parameters of the model: (a) the diffusion coefficient $\sigma$ in front of the perturbation term; (b) the safety loading $\eta$; (c) the shape of the claim size distribution. The function $\phi(\la, x)$ for this risk process is given by
$$
\phi(\la, x) = -p\la + \frac12\si^2\la^2 +  \beta\left[\widehat{B}(\lambda)-1\right],\ \la \ge 0,\ x \in \BR,
$$
where $\widehat{B}(\lambda)=\mathbb{E}(e^{\lambda U})$ denotes the moment generating function (MGF) of the claim amounts distribution. As the expression of $\phi(\la, x)$ actually does not depend on $x$ then
$$
\inf\limits_{x \ge 0}(-\phi(\la, x))=\Phi(\lambda)= p\la - \frac12\si^2\la^2 -  \beta\left[\widehat{B}(\lambda)-1\right],\ \la \ge 0,\ x \in \BR.
$$
The rate of exponential convergence follows from
$$
k=\underset{\{\la\geq0\text{ ; }\widehat{B}(\lambda)<\infty\}}{\text{max}}\,\Phi(\lambda).
$$
The function $\Phi(.)$ is strictly concave as
\begin{equation*}
\Phi''(\lambda)=-\sigma^{2}-\beta\widehat{B}''(\lambda)<0\text{ for all }\la\geq0.
\end{equation*}
It follows that
\begin{equation}\label{eq:LambdaStar}
\la_{\ast}:=\underset{\{\la\geq0\text{ ; }\widehat{B}(\lambda)<\infty\}}{\text{argmax}}\,\Phi(\lambda)
\end{equation}
is solution of the equation
\begin{equation*}
p-\sigma^{2}\la-\beta\widehat{B}'(\lambda)=0,
\end{equation*}
under the constraint $\la^{\ast}\in\{\lambda\geq0\text{ ; }\widehat{B}(\lambda)<\infty\}$. The rate of exponential convergence is then given by
$$
k=\Phi(\la_{\ast})=p\la_{\ast} - \frac12\si^2\la_{\ast}^2 -  \beta\left[\widehat{B}(\la_{\ast})-1\right].
$$
In this example, we compare the rate of convergence $k$ for three claim sizes distribution: the {\it Gamma distribution} $\text{Gamma}(\alpha, \beta)$ with associated probability density function
$$
p(x;  \alpha, \beta) =\begin{cases} \frac{\delta^{\alpha}}{\Gamma(\alpha)}x^{\alpha-1}e^{-\delta x},&\text{ for }t>0\\
0,&\text{ Otherwise},
\end{cases}
$$
the {\it exponential distribution} $\Exp(\delta) = \text{Gamma}(1, \delta)$, and the mixture of exponential distributions $\text{MExp}(p,\delta_1,\delta_2)$ with associated probability density function
$$
p(x;  p,\delta_1,\delta_2)=\begin{cases}
p\delta_1e^{-\delta_1 x}+(1-p)\delta_2e^{-\delta_2 x},&\text{ if }x>0,\\
0,&\text{ otherwise}.
\end{cases}
$$
Let the claim size be distributed as $\text{Gamma}(2,1)$. Table \ref{Tab:ConvergenceRateSafetyLoadingVolatility} gives the rate of exponential convergence for various combinations of values for the safety loading and the volatility.
\begin{table}[h!]\centering
\ra{1.3}
\scriptsize
\begin{tabular}{@{}ll|rrrrrr@{}}\toprule
&&\multicolumn{6}{c}{Safety loading}\\
\cmidrule{3-8}
 \multicolumn{2}{c|}{Volatility}&$\eta=0.05$&$\eta=0.1$&$\eta=0.15$&$\eta=0.2$&$\eta=0.25$&$\eta=0.3$\\
\midrule
$\sigma=$&  0&0.00082 & 0.00319 & 0.00704 & 0.01227 & 0.01881 & 0.02658 \\
& 1&0.0007 & 0.00277 & 0.00613 & 0.01073 & 0.01653 & 0.02345 \\
&2&0.0005 & 0.00197 & 0.00439 & 0.00775 & 0.01201 & 0.01716 \\
& 3&0.00033 & 0.00132 & 0.00297 & 0.00526 & 0.00819 & 0.01174 \\
 &4&0.00023 & 0.00091 & 0.00204 & 0.00361 & 0.00563 & 0.0081 \\
 &5&0.00016 & 0.00064 & 0.00145 & 0.00257 & 0.00402 & 0.00578 \\
& 6&0.00012 & 0.00048 & 0.00107 & 0.0019 & 0.00297 & 0.00427 \\
& 7&0.00009 & 0.00036 & 0.00082 & 0.00145 & 0.00227 & 0.00327 \\
& 8&0.00007 & 0.00029 & 0.00064 & 0.00114 & 0.00178 & 0.00257 \\
& 9&0.00006 & 0.00023 & 0.00052 & 0.00092 & 0.00144 & 0.00207 \\
& 10&0.00005 & 0.00019 & 0.00042 & 0.00075 & 0.00118 & 0.0017 \\
\bottomrule
\end{tabular}\caption{Rate of exponential convergence in the compound Poisson risk model perturbed by a diffusion, with $\text{Gamma}(2,1)$ distributed claim sizes,  and different values for $\sigma$ and $\eta$.}
\label{Tab:ConvergenceRateSafetyLoadingVolatility}
\end{table}

For a given value of the safety loading, the rate of convergences decreases when the volatility increases. Conversely, for a given volatility level, the rate of convergence increases with the safety loading. The first row of Table~\ref{Tab:ConvergenceRateSafetyLoadingVolatility} contains the rates of convergence when $\sigma=0$, associated to the compound Poisson risk model. Figure~\ref{fig:ConvergenceRateSafetyLoadingVolatility} displays the rates of exponential convergence depending of the volatility level for different values of the safety loading: $\eta=0.1,0.2,0.3$.
\begin{figure}[ht]
  \centering
  \includegraphics[width=7cm]{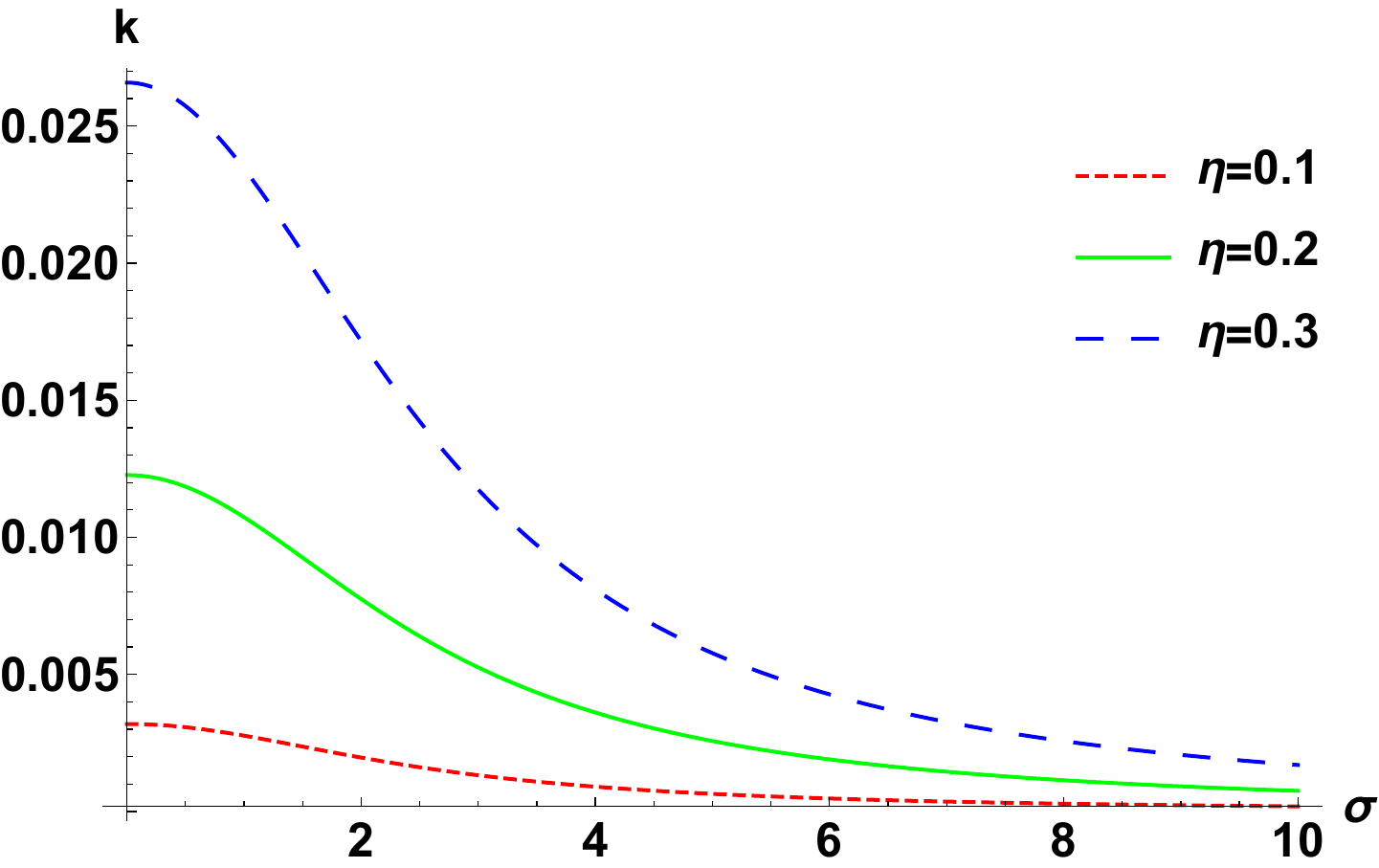}
    \caption{The rate of exponential convergence in the compound Poisson risk model perturbed by a diffusion depending on the volatility, for $\eta=0.1,0.2,0.3$.}
    \label{fig:ConvergenceRateSafetyLoadingVolatility}
\end{figure}
\begin{remark}
Consider the compound Poisson risk model perturbed by a diffusion under constant interest force $i>0$ by assuming that $p(x)=p+ix$, the function $\phi(\la,x)$ then becomes
$$
\phi(\la, x) = -(p+ix)\la + \frac12\si^2\la^2 +  \beta\left[\widehat{B}(\lambda)-1\right],\ \la \ge 0,\ x \in \BR.
$$
Although the function $\phi(\la,x)$ depends on $x$, it is easily seen that
$$
\inf\limits_{x \ge 0}(-\phi(\la, x))=\Phi(\lambda)= p\la - \frac12\si^2\la^2 -  \beta\left[\widehat{B}(\lambda)-1\right],\ \la \ge 0,\ x \in \BR.
$$
The maximization problem is the same as for the compound Poisson risk model perturbed by a diffusion and will lead to the same rate of convergence.
\end{remark}
Let us turn to the study of rate of convergence for different claim sizes distributions. We assume that the claim sizes are either exponentially distributed $\text{Exp}(1/2)$, gamma distributed $\text{Gamma}(2,1)$, or mixture of exponential distributed $\text{MExp}(1/4,3/4,1/4,3/4)$. The mean associated to the claim sizes distributions is the same, but the variance differs:
$$
\Var\left[\text{Gamma}(2,1)\right]<\Var\left[ \text{Exp}(1/2)\right]<\Var\left[ \text{MExp}(3/4,3/4,1/4)\right].
$$
Table \ref{Tab:ConvergenceRateClaimDistribution} contains the values of the rate of exponential convergence over the three claim size distributions.
\begin{table}[h!]\centering
\ra{1.3}
\scriptsize
\begin{tabular}{@{}lll|rcrcr@{}}\toprule
&&&\multicolumn{5}{|c}{Claim Sizes Distributions}\\
\cmidrule{4-8}
Volatility&\multicolumn{2}{c|}{Safety Loadings}&$\text{Exp}(1/2)$&\phantom{abc}&$\text{Gamma}(2,1)$&\phantom{abc}&$\text{MExp}(3/4,3/4,1/4)$\\
\midrule
$\sigma = 0 $&$\eta =$& 0.1 & 0.00238 && 0.00319 && 0.00177 \\
  && 0.2 & 0.00911 && 0.01227 && 0.00668 \\
  && 0.3 & 0.01965 && 0.02658 && 0.01426 \\
  \cmidrule{1-3}
$\sigma = 1$ &$\eta =$& 0.1 & 0.00214 && 0.00277 && 0.00163 \\
  && 0.2 & 0.00824 && 0.01073 && 0.00621 \\
  && 0.3 & 0.01791 && 0.02345 && 0.01335 \\
  \cmidrule{1-3}
$\sigma = 2$ &$\eta =$& 0.1 & 0.00163 && 0.00197 && 0.00132 \\
  && 0.2 & 0.00638 && 0.00775 && 0.00511 \\
  && 0.3 & 0.01405 && 0.01716 && 0.01114 \\
  \cmidrule{1-3}
$\sigma = 3$ &$\eta =$& 0.1 & 0.00116 && 0.00132 && 0.001 \\
  && 0.2 & 0.0046 && 0.00526 && 0.00392 \\
  && 0.3 & 0.01024 && 0.01174 && 0.00865 \\
  \cmidrule{1-3}
$\sigma = 4$ &$\eta =$& 0.1 & 0.00083 && 0.00091 && 0.00074 \\
  && 0.2 & 0.0033 && 0.00361 && 0.00294 \\
 && 0.3 &0.00737 && 0.0081 && 0.00654 \\
 \cmidrule{1-3}
$\sigma = 5$ &$\eta =$& 0.1 & 0.0006 && 0.00064 && 0.00056 \\
  && 0.2 & 0.00241 && 0.00257 && 0.00222 \\
  && 0.3 & 0.00541 && 0.00578 && 0.00496 \\
  \cmidrule{1-3}
$\sigma = 6$ &$\eta =$& 0.1 & 0.00045 && 0.00048 && 0.00043 \\
  && 0.2 & 0.00181 && 0.0019 && 0.0017 \\
  && 0.3 & 0.00407 && 0.00427 && 0.00382 \\
  \cmidrule{1-3}
$\sigma = 7$ &$\eta =$& 0.1 & 0.00035 && 0.00036 && 0.00033 \\
  && 0.2 & 0.0014 && 0.00145 && 0.00134 \\
  && 0.3 & 0.00315 && 0.00327 && 0.003 \\
  \cmidrule{1-3}
$\sigma = 8$ &$\eta =$& 0.1 & 0.00028&& 0.00029 && 0.00027 \\
  && 0.2 & 0.00111 && 0.00114 && 0.00107 \\
  && 0.3 & 0.0025 && 0.00257 && 0.0024 \\
  \cmidrule{1-3}
$\sigma = 9$ &$\eta =$& 0.1 & 0.00022 && 0.00023 && 0.00022 \\
  && 0.2 & 0.0009 && 0.00092 && 0.00087 \\
  && 0.3 & 0.00202 && 0.00207 && 0.00196 \\
  \cmidrule{1-3}
 $\sigma =10$ &$\eta =$& 0.1 & 0.00019 && 0.00019 && 0.00018 \\
  && 0.2 & 0.00074 && 0.00075 && 0.00072 \\
  && 0.3 & 0.00167 && 0.0017 && 0.00162 \\
\bottomrule
\end{tabular}\caption{Rate of exponential convergence in the compound Poisson risk model perturbed by a diffusion for different claim size distribution.}
\label{Tab:ConvergenceRateClaimDistribution}
\end{table}
The fastest convergence occurs in the gamma cases and the slowliest in the mixture of exponential case.
Figure \ref{fig:ConvergenceRateVolatilitySafetyLoadingClaims} displays the evolution of the rate of exponential convergence depending on the safety loading and the diffusion parameter for the different assumption over the claim sizes.
\begin{figure}[h!]
\centering
\subfigure[The rate of exponential convergence depending on the safety loading and diffusion $\sigma=2$.]
{
\includegraphics[width=6cm]{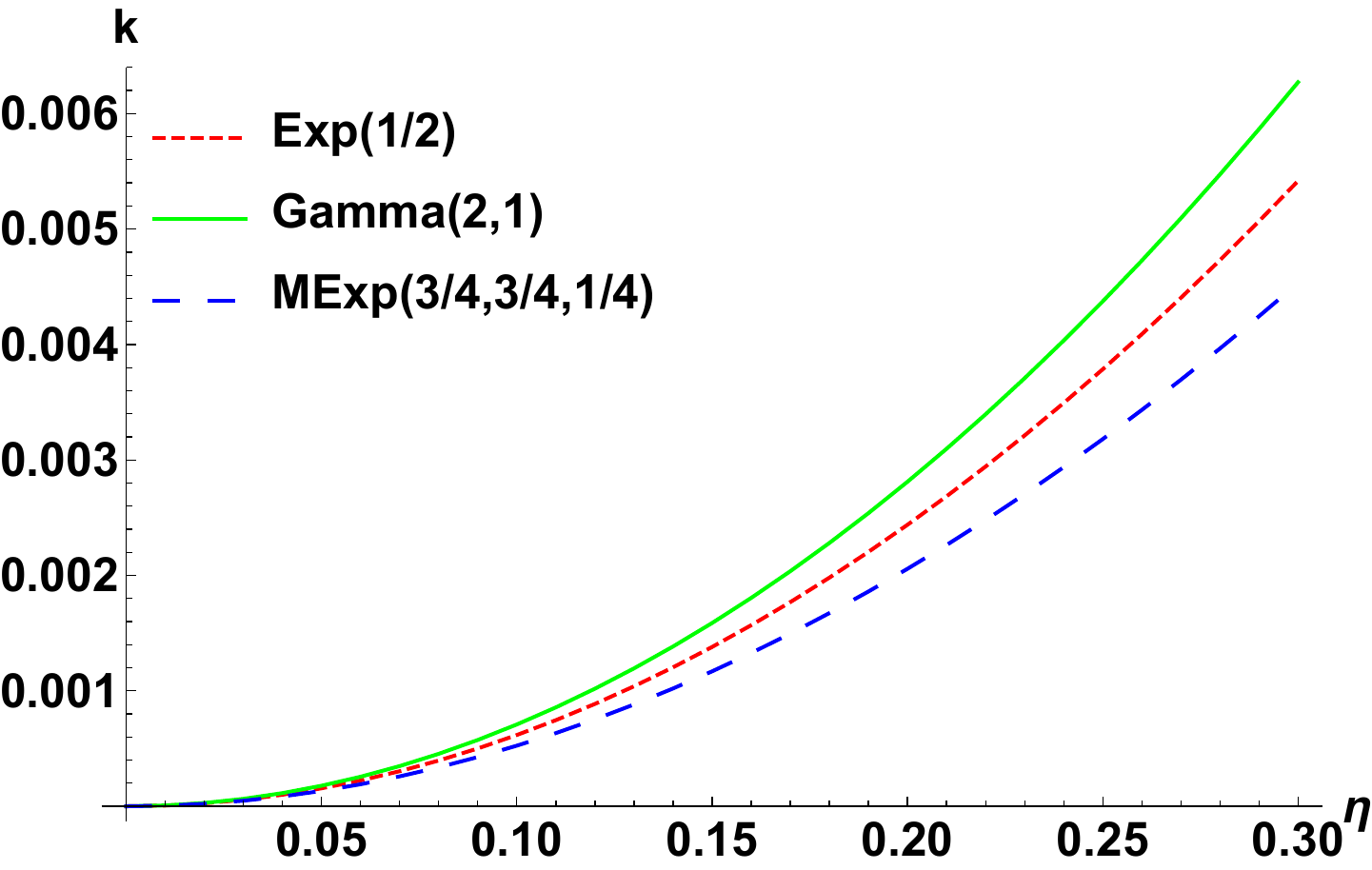}
\label{fig:ConvergenceRateSafetyLoadingClaims}
}
\subfigure[The rate of exponential convergence depending on the volatility and safety loading $\eta=0.1$.]
{
\includegraphics[width=6cm]{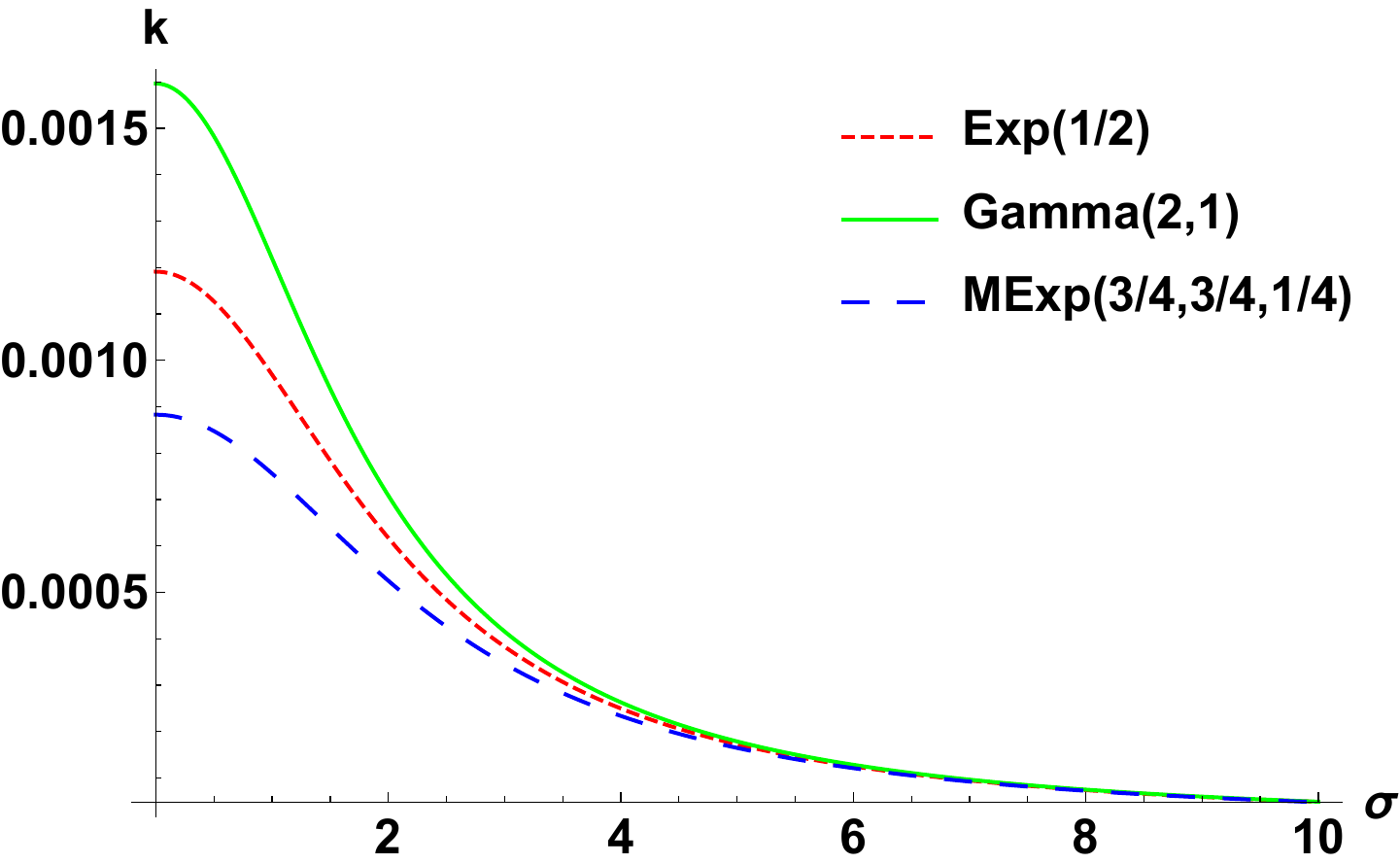}
\label{fig:ConvergenceRateVolatilityClaims}
}
\caption{The rate of exponential convergence in the compound Poisson risk model perturbed by a diffusion for different claim sizes distributions}
\label{fig:ConvergenceRateVolatilitySafetyLoadingClaims}
\end{figure}
In the wake of this numerical study, we may conclude that the speed of convergence depends on the variance of the process. Increasing the variance through the claim sizes distribution or via the diffusion component makes the convergence toward the stationary distribution slower.

\subsection{L\'evy driven risk process.}
In this subsection, we compare the rate of exponential convergence of the ruin probabilities when the liability of the insurance company is modeled by a \textit{gamma process} and an \textit{inverse Gaussian L\'evy process}. The L\'evy measure of a \textit{gamma process}, $\text{GammaP}(\alpha,\beta)$, is given by
\begin{equation}\label{eq:LevyMeasureGammaProcess}
\mu(\text{d}x)=\frac{\alpha e^{-\beta x}}{x},\text{ for }x>0,
\end{equation}
where  $\alpha,\,\beta>0$. Its L\'evy exponent is
\begin{equation}\label{eq:LevyExpGammaProcess}
\kappa(\lambda)=\alpha\ln\left(\frac{\beta}{\beta-\la}\right),\text{ for }\la\in[0,\beta).
\end{equation}
The function $\Phi(\cdot)$ is strictly concave as
\begin{equation*}
\Phi''(\lambda)=-\sigma^{2}-\frac{\alpha}{(\beta-\lambda)^2}<0.
\end{equation*}
It follows that $\la_{\ast}$ is the solution of the equation
\begin{equation*}
p-\sigma^{2}\la-\frac{\alpha}{\beta-\lambda}=0.
\end{equation*}
The rate of exponential convergence is then given by
$$
k=\Phi(\la_{\ast})=p\la_{\ast} - \frac12\si^2\la_{\ast}^2 - \alpha\ln\left(\frac{\beta}{\beta-\la_\ast}\right).
$$
The L\'evy measure associated to the \textit{inverse Gaussian L\'evy process}, $\text{IGP}(\gamma)$, is defined as
\begin{equation}\label{eq:LevyMeasureInverseGaussianProcess}
\mu(\text{d}x)=\frac{1}{\sqrt{2\pi}x^{3/2}}e^{-x\gamma^{2}/2},\text{ for }x>0.
\end{equation}
where $\gamma>0$. Its L\'evy exponent is
\begin{equation}\label{eq:LevyExpIGProcess}
\kappa(\lambda)=\gamma-\sqrt{\gamma^{2}-2\la},\text{ for }\la\in[0,\gamma^{2}/2).
\end{equation}
The function $\Phi$ is strictly concave as
\begin{equation*}
\Phi''(\lambda)=-\sigma^{2}-(\gamma^{2}-2\lambda)^{-3/2}<0
\end{equation*}
It follows that $\la_{\ast}$ is the solution of the equation
\begin{equation*}
p-\sigma^{2}\la-\frac{1}{\sqrt{\gamma^{2}-2\lambda}}=0,
\end{equation*}
The rate of exponential convergence is then given by
$$
k=\Phi(\la_{\ast})=p\la_{\ast} - \frac12\si^2\la_{\ast}^2 - \gamma+\sqrt{\gamma^2-2\la_\ast}.
$$
We set $\gamma=1$, $\alpha=1/2$, $\beta=1/2$, to match the first moment of the liabilities in both risk model at time $t=1$. Table \ref{Tab:ConvergenceRateLevyDriven} contains the value of the exponential rate of convergence when the liability of the insurance company is governed by a \textit{gamma process} or an \textit{inverse Gausian L\'evy process} depending on the safety loading and the volatility of the diffusion.
\begin{table}[h!]\centering
\ra{1.3}
\scriptsize
\begin{tabular}{@{}lll|rcr@{}}\toprule
&&&\multicolumn{3}{|c}{L\'evy processes}\\
  \cmidrule{4-6}
Volatility&\multicolumn{2}{c|}{Safety Loadings}&\text{GammaP}(1/2,1/2)&\phantom{abc}&\text{IGP}(1)\\
\midrule
$\sigma=0$ &$\eta =$& 0.1 & 0.02617 && 0.05 \\
  && 0.2 & 0.05442 && 0.1 \\
  && 0.3 & 0.08441 && 0.15 \\
    \cmidrule{1-3}
$\sigma= 1$ &$\eta =$& 0.1 & 0.01809 & &0.0271 \\
  && 0.2 & 0.03882 && 0.05806 \\
  && 0.3 & 0.06189 && 0.09238 \\
  \cmidrule{1-3}
$\sigma= 2$&$\eta =$ & 0.1 & 0.00921 && 0.01104 \\
 && 0.2 & 0.02013 && 0.02412 \\
  && 0.3 & 0.03272 & &0.03923 \\
  \cmidrule{1-3}
$\sigma= 3$&$\eta =$ & 0.1 & 0.00503 && 0.00552 \\
  && 0.2 & 0.01101 && 0.01207 \\
  && 0.3 & 0.01794 && 0.01965 \\
  \cmidrule{1-3}
$\sigma= 4$&$\eta =$ & 0.1 & 0.00307 && 0.00324 \\
  && 0.2 & 0.00671 && 0.00709 \\
  && 0.3 & 0.01094 && 0.01153 \\
  \cmidrule{1-3}
$\sigma= 5$&$\eta =$ & 0.1 & 0.00204 && 0.00212 \\
  && 0.2 & 0.00447 && 0.00463 \\
  && 0.3 & 0.00727 && 0.00753 \\
  \cmidrule{1-3}
$\sigma= 6$&$\eta =$ & 0.1 & 0.00145 && 0.00149 \\
  && 0.2 & 0.00317 & &0.00325 \\
  && 0.3 & 0.00516 & &0.00529 \\
  \cmidrule{1-3}
$\sigma= 7$&$\eta =$ & 0.1 & 0.00108 && 0.0011 \\
  && 0.2 & 0.00236 && 0.0024 \\
  && 0.3 & 0.00384 && 0.00391 \\
  \cmidrule{1-3}
$\sigma= 8$&$\eta =$ & 0.1 & 0.00083 && 0.00085 \\
  && 0.2 & 0.00182 & &0.00185 \\
  && 0.3 & 0.00296 && 0.00301 \\
  \cmidrule{1-3}
$\sigma= 9 $&$\eta =$& 0.1 & 0.00066 && 0.00067 \\
  && 0.2 & 0.00145 && 0.00146 \\
&& 0.3 & 0.00236 && 0.00238 \\
\cmidrule{1-3}
$\sigma= 10$&$\eta =$ & 0.1 & 0.00054 && 0.00054 \\
  && 0.2 & 0.00118 && 0.00119 \\
  && 0.3 & 0.00192 & &0.00193 \\
\bottomrule
\end{tabular}\caption{Rate of exponential convergence in L\'evy driven risk models.}
\label{Tab:ConvergenceRateLevyDriven}
\end{table}
Figure \ref{fig:ConvergenceRateSafetyLoadingLevy} displays the rates of exponential convergence for the considered L\'evy driven risk models.
\begin{figure}[h!]
\centering
\subfigure[The rate of exponential convergence depending on the safety loading, and volatility $\sigma=1$.]
{
\includegraphics[width=6cm]{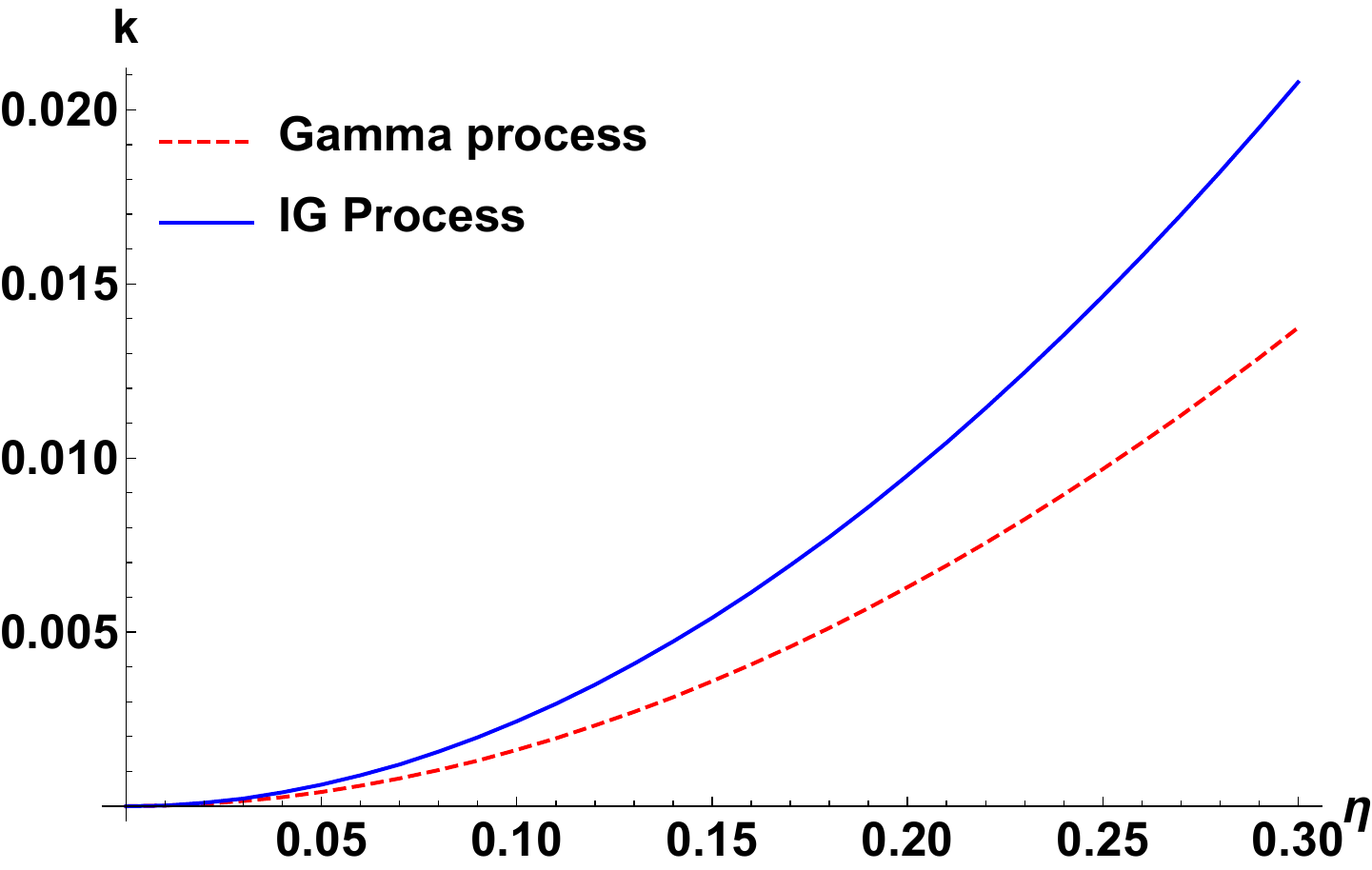}
\label{fig:ConvergenceRateSafetyLoadingLevy}
}
\subfigure[The rate of exponential convergence depending on the volatility and safety loading $\eta=0.2$.]
{
\includegraphics[width=6cm]{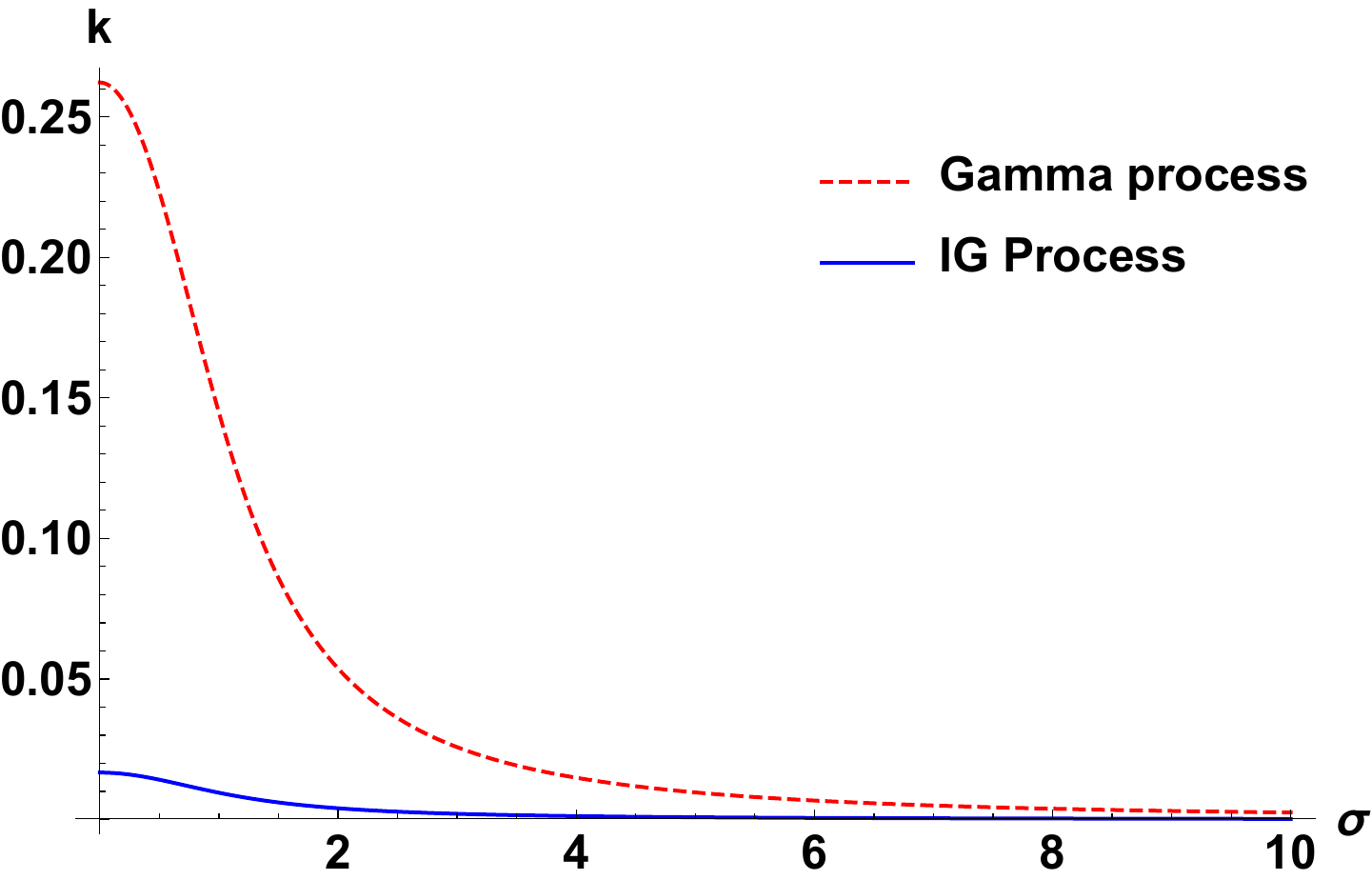}
\label{fig:ConvergenceRateVolatilityLevy}
}
\caption{The rate of exponential convergence for L\'evy driven risk processes.}
\label{fig:ConvergenceRateLevy}
\end{figure}
We observe that the impact of the volatility and the safety loading on the convergence rate remains the same as in the compound Poisson case. The rate of exponential convergence is noticeably greater when the liability of the insurance company follows an \textit{inverse Gaussian L\'evy process}.

\section{Proof of Theorem~\ref{thm:main-1}}\label{sec:ProofMainThm}
 If $Y$ were a reflected jump-diffusion with a.s. finitely many jumps in finite time, and with positive diffusion coefficient, then we could directly apply \cite[Theorem 4.1, Theorem 4.3]{MyOwn12}, and complete the proof of Theorem~\ref{thm:main-1}. However, we might have: (a) zero diffusion coefficient $\si(x) = 0$ for some $x$; (b) infinite L\'evy measure $\mu$, that is, infinitely many jumps in finite time horizon.

\smallskip

In the proof of \cite[Theorem 3.2]{MyOwn12}, we used the following property: for all $t > 0$, $x \in \BR_+$, and $A \subseteq \BR_+$ of positive Lebesgue measure, we have $Q^t(x, A) > 0$. This property might not hold for the case $\si(x) = 0$ for some $x \in \BR_+$. We bypass this difficulty via the following method: approximating the reflected jump-diffusion $Y$ by a ``regular'' reflected jump-diffusion, where $\si(x) > 0$ for $x \in \BR_+$, and the L\'evy measure is finite.

\smallskip
%
%In the first part of this proof, we assume the drift coefficient $p_*$ is Lipschitz. In the second part, we deal with the general (sublinear) drift coefficient by approximating it with Lipschitz drifts.

For an $\eps > 0$, let $Y_{\eps} = (Y_{\eps}(t),\, t \ge 0)$ be the reflected jump-diffusion on $\BR_+$, with drift coefficient $p_*$, diffusion coefficient $\si_{\eps}(\cdot) = \si(\cdot) + \eps$, and jump measure $\mu_{\eps}(\cdot) = \mu(\cdot\cap[\eps, \eps^{-1}])$. Note that this is a reflected jump-diffusion with positive diffusion coefficient, and with finite L\'evy measure: $\si_{\eps}(y) > 0$ for all $y \in \BR_+$, and $\mu_{\eps}(\BR_+) < \infty$. Therefore, we can apply results of \cite{MyOwn12} to this process. For $x \in \BR_+$, let
$$
\phi_{\eps}(x, \la) := p_{\ast}(x)\la + \frac12\si^2_{\eps}(x)\la^2 + \int_{\eps}^{\eps^{-1}}\left(e^{\la y} - 1\right)\,\mu_{\eps}(\md y).
$$
For every $x \ge 0$, we have:
\begin{equation}
\label{eq:difference-k}
\phi(x, \la) -\phi_{\eps}(x, \la) = -\left[\eps\si_{\eps}(x) + \frac12\eps^2\right]\la^2 + \left(\int_0^{\eps} + \int_{\eps^{-1}}^{\infty}\right)\left(e^{\la y} - 1\right)\,\mu(\md y).
\end{equation}
Recall also that
\begin{equation}
\label{eq:finite-exp-moment}
\int_0^{\infty}\left(e^{\la y} - 1\right)\,\mu(\md y) < \infty.
\end{equation}
Combining~\eqref{eq:difference-k} with~\eqref{eq:finite-exp-moment} and the boundedness of $\si$ from Assumption \ref{as:1}, we have:
\begin{equation}
\label{eq:uniform-convergence-k}
\sup\limits_{x \ge 0}\left|\phi_{\eps}(x, \la) - \phi(x, \la)\right| \to 0,\ \eps \downarrow 0.
\end{equation}
By our assumptions,
\begin{equation}
\label{eq:k-sup-negative}
\sup\limits_{x \ge 0}\phi(x, \la) = -\Phi(\la) < 0.
\end{equation}
From~\eqref{eq:uniform-convergence-k}, we have:
\begin{equation}
\label{eq:k-sup-convergence}
-\sup\limits_{x \ge 0}\phi_{\eps}(x, \la) =: \Phi_{\eps}(\la) \to \Phi(\la).
\end{equation}
From~\eqref{eq:k-sup-convergence} and~\eqref{eq:k-sup-negative}, we conclude that  there exists an $\eps_0 > 0$ such that for $\eps \in [0,\eps_0]$, $\Phi_{\eps}(\la) > 0$. Apply \cite[Theorem 4.3]{MyOwn12} to prove the statement of Theorem~\ref{thm:main-1} for the process $Y_{\eps}$. For consistency of notation, denote $Y_0 := Y$. There exists a unique  stationary distribution $\pi_{\eps}$ for $Y_{\eps}$, which satisfies $(\pi_{\eps}, V_{\la}) < \infty$; and the transition kernel $Q_{\eps}^t(x, \cdot)$ of this process $Y_{\eps}$ satisfies
\begin{equation}
\label{eq:exp-ergodicity-eps}
\norm{Q_{\eps}^t(x, \cdot) - \pi_{\eps}(\cdot)}_{V_{\la}} \le \left[V_{\la}(x) + (\pi_{\eps}, V_{\la})\right]e^{-\Phi_{\eps}(\la)t}.
\end{equation}
We would like to take the limit $\eps \downarrow 0$ in~\eqref{eq:exp-ergodicity-eps}. To this end, let us introduce some new notation. Take a smooth $C^{\infty}$ function $\theta : \BR_+ \to \BR_+$ which is nondecreasing, and satisfies
$$
\theta(x) =
\begin{cases}
0,\ x \le s_-;\\
x,\ x \ge s_+;
\end{cases}
\ \ \theta(x) \le x,
$$
for some fixed $s_+ > s_- > 0$. The function $\theta$ is Lipschitz on $\BR_+$: there exists a constant $C(\theta) > 0$ such that
\begin{equation}
\label{eq:theta-Lipschitz}
|\theta(s_1) - \theta(s_2)| \le C(\theta)|s_1 - s_2|\ \mbox{for all}\ s_1, s_2 \in \BR_+.
\end{equation}
Next, define
$$
\tilde{V}_{\la}(x) = V_{\la}(\theta(x)) = e^{\la\theta(x)}.
$$
The process $Y_{\eps}$ has the generator $\CL_{\eps}$, given by the formula
$$
\CL_{\eps}f(x) = p_{\ast}(x)f'(x) + \frac12\si_{\eps}^2(x)f''(x) + \int_{\eps}^{\eps^{-1}}\left[f(x+y) - f(x)\right]\,\mu(\md y)
$$
for $f \in C^2(\BR_+)$ with $f'(0) = 0$. Repeating calculations from \cite[Theorem 3.2]{MyOwn12} with minor changes, we get:
\begin{equation}
\label{eq:Lyapunov-example}
\CL_{\eps}\tilde{V}_{\la}(x) \le -\Phi_{\eps}(\la)\tilde{V}_{\la}(x) + c_{\eps}1_{[0, s_+]}(x),\ x \in \BR_+,
\end{equation}
with the constant
\begin{equation}
\label{eq:c-constant}
c_{\eps} := \max\limits_{x\in[0, s_+]}\left[\CL_{\eps}\tilde{V}_{\la}(x) + \phi_{\eps}(\la,x)\tilde{V}_{\la}(x)\right].
\end{equation}

\begin{lemma}
$\varlimsup_{\eps \downarrow 0}(\pi_{\eps}, V_{\la}) < \infty$.
\label{lemma:finite-limsup}
\end{lemma}

\begin{proof}
The functions $V_{\la}$ and $\tilde{V}_{\la}(x)$ are of the same order, in the sense that
\begin{equation}
\label{eq:equiv}
0 < \inf\limits_{x \ge 0}\frac{\tilde{V}_{\la}(x)}{V_{\la}(x)} \le \sup\limits_{x \ge 0}\frac{\tilde{V}_{\la}(x)}{V_{\la}(x)} < \infty.
\end{equation}
Therefore, it suffices to show that
\begin{equation}
\label{eq:tilde-finite-limsup}
\varlimsup\limits_{\eps \downarrow 0}(\pi_{\eps}, \tilde{V}_{\la}) < \infty.
\end{equation}
Apply the probability measure $\pi_{\eps}$ to both sides of the inequality~\eqref{eq:Lyapunov-example}. This probability measure is stationary; therefore, the left-hand side of~\eqref{eq:Lyapunov-example} becomes $(\pi_{\eps}, \CL_{\eps}\tilde{V}_{\la}) = 0$. Therefore,
$$
-\Phi_{\eps}(\la)\bigl(\pi_{\eps}, \tilde{V}_{\la}\bigr) + c_{\eps}\bigl(\pi_{\eps}, 1_{[0, s_+]}\bigr) \ge 0.
$$
Since $(\pi_{\eps}, 1_{[0, s_+]}) = \pi_{\eps}([0, s_+]) \le 1$, we get:
\begin{equation}
\label{eq:upper-bound}
\bigl(\pi_{\eps}, \tilde{V}_{\la}\bigr) \le \frac{c_{\eps}}{\Phi_{\eps}(\la)}.
\end{equation}
From~\eqref{eq:k-sup-convergence} and~\eqref{eq:upper-bound}, to show~\eqref{eq:tilde-finite-limsup}, it suffices to show that
\begin{equation}
\label{eq:c-eps}
\varlimsup\limits_{\eps\downarrow 0}c_{\eps} < \infty.
\end{equation}
This, in turn, would follow from~\eqref{eq:c-constant},~\eqref{eq:k-sup-convergence},  and the following relation:
\begin{equation}
\label{eq:uniform-conv-of-generators}
\CL_{\eps}\tilde{V}_{\la}(x) \to \CL\tilde{V}_{\la}(x),\ \ \mbox{uniformly on}\ \ [0, s_+].
\end{equation}
We can express the difference of generators as
\begin{align}
\label{eq:difference-generators}
\begin{split}
\CL_{\eps}&\tilde{V}_{\lambda}(x) - \CL\tilde{V}_{\lambda}(x) \\ & = \frac12\left(\si_{\eps}^2(x) - \si^2(x)\right)f''(x) - \left(\int_0^{\eps}+\int_{\eps^{-1}}^{\infty}\right)\left[\tilde{V}_{\la}(x+y) - \tilde{V}_{\la}(x)\right]\,\mu(\md y).
\end{split}
\end{align}
The first term in the right-hand side of~\eqref{eq:difference-generators} is equal to
$\frac12(2\eps\si(x) + \eps^2)f''(x)$. Since $\si$ is bounded, this term converges to $0$ as $\eps \downarrow 0$ uniformly on $[0, s_+]$. It suffices to prove that the second term converges to zero as well. For all $x, y \ge 0$, using~\eqref{eq:theta-Lipschitz}, we have:
\begin{align}
\label{eq:elementary-est}
\begin{split}
& 0 \le  \tilde{V}_{\la}(x + y) - \tilde{V}_{\la}(x) = e^{\la\theta(x+y)} - e^{\la\theta(x)} \\ & = e^{\la\theta(x)}\left[e^{\la(\theta(x+y) - \theta(x))} - 1\right] \le \tilde{V}_{\la}(x)\left[e^{\la C(\theta)y}  - 1\right].
\end{split}
\end{align}
Changing the parameter $s_-$ and letting $s_- \downarrow 0$, we have: $\theta(x) \to x$ uniformly on $\BR_+$. Therefore, we can make the Lipschitz constant $C(\theta)$ as close to $1$ as necessary. Also, note that for $\la'$ in some neighborhood of $\la$, we have:
\begin{equation}
\label{eq:finite-exp-moment-nbhd}
\int_0^{\infty}\left(e^{\la'x} - 1\right)\,\mu(\md x) < \infty.
\end{equation}
Combining~\eqref{eq:finite-exp-moment-nbhd} with~\eqref{eq:elementary-est}, using that $\sup_{x \in [0, s_+]}\tilde{V}_{\la}(x) < \infty$, and making $C(\theta)$ close enough to $1$, we complete the proof that the second term in the right-hand side of~\eqref{eq:difference-generators} tends to $0$ as $\eps \downarrow 0$. This completes the proof of~\eqref{eq:uniform-conv-of-generators}, and with it that of~\eqref{eq:c-eps} and Lemma~\ref{lemma:finite-limsup}.
\end{proof}

Now, we state a fundamental lemma, and complete the proof of Theorem~\ref{thm:main-1} assuming that this lemma is proved. The proof is postponed until the end of this section.

\begin{lemma}
Take a version $\tilde{Y}_{\eps}$ of the reflected jump-diffusion $Y_{\eps}$, starting from $y_{\eps} \ge 0$, for $\eps \ge 0$. If $y_{\eps} \to y_0$, then we can couple $\tilde{Y}_{\eps}$ and $\tilde{Y}_0$ so that for every $T \ge 0$,
$$
\lim\limits_{\eps \downarrow 0}\ME\sup\limits_{0 \le t \le T}\bigl|\tilde{Y}_{\eps}(t) - \tilde{Y}_{0}(t)\bigr|^2 = 0.
$$
\label{lemma:fundamental}
\end{lemma}

Since $V_{\la}(\infty) = \infty$, Lemma~\ref{lemma:finite-limsup} implies tightness of the familly $(\pi_{\eps})_{\eps \in (0, \eps_0]}$ of probability measures. Now take a stationary version $\ol{Y}_{\eps}$ of the reflected jump-diffusion $Y_{\eps}$: for every $t \ge 0$, let $\ol{Y}_{\eps}(t) \sim \pi_{\eps}$. Take a sequence $(\eps_n)_{n \ge 1}$ such that $\eps_n \downarrow 0$ as $n \to \infty$, and $\pi_{\eps_n} \Ra \pi_0$ (where $\Ra$ stands for weak convergence) for some probability measure $\pi_0$ on $\BR_+$. It follows from Lemma~\ref{lemma:fundamental} that for every $t \ge 0$, we have: $\ol{Y}_{\eps_n}(t) \Ra \ol{Y}_0(t)$ as $n \to \infty$, where $\ol{Y}_0$ is a stationary version of the reflected jump-diffusion $Y_0$: that is, $\ol{Y}_0(t) \sim \pi_0$ for every $t \ge 0$. In other words, we proved that the reflected jump-diffusion $Y_0$ has a stationary distribution $\pi_0$.

\smallskip

Next, take a measurable function $g : \BR_+ \to \BR$ such that $|g(x)| \le V_{\la}(x)$ for all $x \in \BR_+$.

\begin{lemma} $\left(\pi_{\eps_n}, g\right) \to (\pi_0, g)$ as $n \to \infty$.
\label{lemma:DCT}
\end{lemma}

\begin{proof} The function $\Phi$ is a supremum of a family of functions $-\phi(\cdot, x)$, which are continuous in $\la$. Therefore, $\Phi$ is lower semicontinuous, and the set $\{\la > 0\mid \Phi(\la) > 0\}$ is open. Apply Lemma~\ref{lemma:finite-limsup} to some $\la' > \la$ (which exists by the observation above). Then we get:
$$
\varlimsup\limits_{\eps \downarrow 0}\left(\pi_{\eps_n}, V_{\la'}\right) < \infty.
$$
Note also that $|g(x)|^{\la'/\la} \le \left[V_{\la}(x)\right]^{\la'/\la} = V_{\la'}(x)$ for all $x \ge 0$. Therefore, the family $(\pi_{\eps}g^{-1})_{\eps \in (0, \eps_0]}$ of probability distributions is uniformly integrable. Uniform integrability plus a.s. convergence imply convergence of expected values. Thus we complete the proof of Lemma~\ref{lemma:DCT}.
\end{proof}

For all $\eps \ge 0$, take a copy $Y^{\eps}$ of $Y_{\eps}$ starting from the same initial point $x \in \BR_+$.

\begin{lemma} For every $t \ge 0$, we have: $\ME g(Y^{\eps}(t)) \to \ME g(Y^0(t))$ as $\eps \downarrow 0$.
\label{lemma:DCT-time}
\end{lemma}

\begin{proof} Following calculations in the proof of \cite[Theorem 3.2]{MyOwn12}, we get:
\begin{equation}
\label{eq:Lyapunov-integral}
\ME \tilde{V}_{\la}(Y^{\eps}(t)) - \tilde{V}_{\la}(x) \le \int_0^t\left[-\Phi_{\eps}(\la)\tilde{V}_{\la}(Y^{\eps}(s)) + c_{\eps}1_{[0, s_+]}(s)\right]\,\md s \le c_{\eps}t.
\end{equation}
Therefore, from~\eqref{eq:Lyapunov-integral} we have:
\begin{equation}
\label{eq:bdd-new}
\varlimsup\limits_{\eps\downarrow 0}\ME \tilde{V}_{\la}(Y^{\eps}(t)) < \infty.
\end{equation}
From~\eqref{eq:equiv}, ~\eqref{eq:bdd-new} holds for $V_{\la}$ in place of $\tilde{V}_{\la}$. This is also true for $\la' > \la$ slightly larger than $\la$.
 Applying the same uniform integrability argument as in the proof of Lemma~\ref{lemma:DCT}, we complete the proof of Lemma~\ref{lemma:DCT-time}.
\end{proof}

Finally, let us complete the proof of Theorem~\ref{thm:main-1}. From~\eqref{eq:exp-ergodicity-eps}, we have:
\begin{equation}
\label{eq:final-step}
\left|\ME g(Y^{\eps}(t)) - (\pi_{\eps}, g)\right| \le \left[V_{\la}(x) + \left(\pi_{\eps}, V_{\la}\right)\right]e^{-\Phi_{\eps}(\la)t}.
\end{equation}
Taking $\eps = \eps_n$ and letting $n \to \infty$ in~\eqref{eq:final-step}, we use Lemma~\ref{lemma:DCT} and~\ref{lemma:DCT-time} to conclude that
\begin{equation}
\label{eq:final-formula}
\left|\ME g(Y^{0}(t)) - (\pi_{0}, g)\right| \le \left[V_{\la}(x) + \left(\pi_{0}, V_{\la}\right)\right]e^{-\Phi(\la)t}.
\end{equation}
Take the supremum over all functions $g : \BR_+ \to \BR$ which satisfy $|g(x)| \le V_{\la}(x)$ for all $x \in \BR_+$, and complete the proof of Theorem~\ref{thm:main-1} for Lipschitz $p_*$.

\subsection{Proof of Lemma~\ref{lemma:fundamental}} Let us take a probability space with independent Brownian motion $W$ and L\'evy process $L$, and let $L_{\eps}$ be  a subordinator process with L\'evy measure $\mu_{\eps}$, obtained from $L$ by eliminating all jumps of size less than $\eps$ and greater than $\eps^{-1}$. For consistency of notation, let $L_0 := 0$. For every $\eps \ge 0$, we can represent
\begin{equation}
\label{eq:reflected-formula}
\tilde{Y}_{\eps}(t) = y_{\eps} + \int_0^tp_{\ast}(\tilde{Y}_{\eps}(s))\,\md s + \int_0^t\si_{\eps}(\tilde{Y}_{\eps}(s))\,\md W(s) + L_{\eps}(t) + N_{\eps}(t),\ t \ge 0.
\end{equation}
Here, $N_{\eps}$ is a nondecreasing right-continuous process with left limits, with $N_{\eps}(0) = 0$, which can increase only when $\tilde{Y}_{\eps} = 0$. We can rewrite~\eqref{eq:reflected-formula} as
\begin{equation}
\label{eq:thru-X}
\tilde{Y}_{\eps}(t) = \mathcal X_{\eps}(t) + \int_0^tp_{\ast}(\tilde{Y}_{\eps}(s))\,\md s + \int_0^t\si(\tilde{Y}_{\eps}(s))\,\md W(s) + N_{\eps}(t),\ t \ge 0.
\end{equation}
Here, we introduce a new piece of notation:
\begin{equation}
\label{eq:cal-X}
\mathcal X_{\eps}(t) = y_{\eps} + L_{\eps}(t) + \eps W(t),\ t \ge 0.
\end{equation}
The process $L(\cdot) - L_{\eps}(\cdot)$ is nondecreasing. By Assumption 2.3, as $\eps \downarrow 0$, for every $T > 0$,
\begin{equation}
\label{eq:conv-of-L}
\ME\sup\limits_{0 \le t \le T}\left|L(t) - L_{\eps}(t)\right|^2 = \ME\left(L(T) - L_{\eps}(T)\right)^2 = T\left(\int_0^{\eps} + \int_{\eps^{-1}}^{\infty}\right)\,x^2\,\mu(\md x)   \to 0.
\end{equation}
From~\eqref{eq:cal-X} and~\eqref{eq:conv-of-L}, we have:
\begin{equation}
\label{eq:eps-0}
\ME\sup\limits_{0 \le t \le T}\left|\mathcal X_0(t) - \mathcal X_{\eps}(t)\right|^2 \to 0,\ \eps \downarrow 0.
\end{equation}
Fix time horizon $T > 0$, and consider the space $\mathcal E_T$ of all right-continuous adapted processes $Z = (Z(t),\, 0 \le t \le T)$ with left limits such that
$$
\norm{Z}_{2, T}^2 := \ME\sup\limits_{0 \le t \le T}Z^2(t) < \infty.
$$
This is a Banach space with norm $\norm{\cdot}_{2, T}$. Fix an $\mathcal X \in \mathcal E_T$. Let us introduce two mappings
$\mathcal P_{\mathcal X},\, \mathcal S : \mathcal E_T \to \mathcal E_T$: The mapping $\mathcal P_{\mathcal X}$ is given by
$$
\mathcal P_{\mathcal X}(Z)(t) = \mathcal X(t) + \int_0^tp_{\ast}(Z(s))\,\md s + \int_0^t\si(Z(s))\,\md W(s),\ 0 \le t \le T.
$$
Whereas $\mathcal S$ is the classic Skorohod mapping:
$$
\mathcal S(Z)(t) = Z(t) + \sup\limits_{0 \le s \le t}(Z(s))_-,\ 0 \le t \le T,
$$
where $(a)_{-}:=\max(-a, 0)$ for any $a\in\mathbb{R}$. For any $\mathcal X \in \mathcal E_T$, let $\mathcal R_{\mathcal X} := \mathcal S\circ\mathcal P_{\mathcal X}$. Then we can represent~\eqref{eq:thru-X} as
\begin{equation}
\label{eq:fixed-point-equation}
Y_{\eps} = (\mathcal S \circ \mathcal P_{\mathcal X_{\eps}}(Y_{\eps})) = \mathcal R_{\mathcal X_{\eps}}(Y_{\eps}).
\end{equation}
It is straightforward to show, using Lipschitz properties of $p_{\ast}$ and $\si$, that these mappings indeed map $\mathcal E_T$ into $\mathcal E_T$. Moreover, a classic result is that $\mathcal S$ is $1$-Lipschitz. See, for example, \cite{Whitt}. Assume $C(p_{\ast})$ and $C(\si)$ are Lipschitz constants for functions $p_{\ast}$ and $\si$.

\begin{lemma} For $\mathcal X, \mathcal X', \mathcal Z, \mathcal Z' \in \mathcal E_T$, the following Lipschitz property holds with constant
\begin{equation}
\label{eq:C-T}
C_T := C(p_\ast)T + 2C(\si)T^{1/2}.
\end{equation}
\begin{equation}
\label{eq:Lipschitz-general}
\norm{\mathcal R_{\mathcal X}(\mathcal Z) - \mathcal R_{\mathcal X'}(\mathcal Z')}_{2, T} \le C_T\norm{\mathcal Z - \mathcal Z'}_{2, T} + \norm{\mathcal X - \mathcal X'}_{2, T}.
\end{equation}
\label{lemma:Lipschitz-general}
\end{lemma}

\begin{proof} Since $\mathcal S$ is $1$-Lipschitz, it suffices to show~\eqref{eq:Lipschitz-general} for $\mathcal P_{\mathcal X}$ instead of $\mathcal R_{\mathcal X}$. We can express the difference between $\mathcal P_{\mathcal X}(\mathcal Z)$ and $\mathcal P_{\mathcal X'}(\mathcal Z')$ as follows: for $t \in [0, T]$,
\begin{align}
\label{eq:big-difference}
\begin{split}
\mathcal P_{\mathcal X}&(\mathcal Z)(t) - \mathcal P_{\mathcal X'}(\mathcal Z')(t)  = \mathcal X(t) - \mathcal X'(t) \\ & + \int_0^t\left[p_\ast(\mathcal Z(s)) - p_\ast(\mathcal Z'(s))\right]\,\md s + \int_0^t\left[\si(\mathcal Z(s)) - \si(\mathcal Z'(s))\right]\,\md W(s).
\end{split}
\end{align}
Denoting by $I$ and $M$ the second and third terms in the right-hand side of~\eqref{eq:big-difference}, we have:
\begin{align}
\label{eq:norm-diff}
\norm{\mathcal P_{\mathcal X}(\mathcal Z)(t) - \mathcal P_{\mathcal X'}(\mathcal Z')(t)}_{2, T} \le \norm{\mathcal X - \mathcal X'}_{2, T} + \norm{I}_{2, T} + \norm{M}_{2, T}.
\end{align}
The norm $\norm{I}_{2, T}$ is estimated in a straightforward way using the Lipschitz property of $\sigma$:
\begin{align}
\label{eq:term-I}
\begin{split}
\norm{I}_{2, T}^2 &= \ME\sup\limits_{0 \le t \le T}I^2(t)\le \ME\sup\limits_{0 \le t \le T}\left(\int_0^tC(p_\ast)\left[\mathcal Z(s) - \mathcal Z'(s)\right]\,\md s\right)^2\\& \le T^2C^2(p_\ast)\cdot \ME\sup\limits_{0 \le s \le T}\left[\mathcal Z(s) - \mathcal Z'(s)\right]^2  = T^2C^2(p_\ast)\norm{\mathcal Z - \mathcal Z'}^2_{2, T}.
\end{split}
\end{align}
Finally, the norm $\norm{M}_{2, T}$ can be estimated using the martingale inequalities:
\begin{align}
\label{eq:term-M}
\begin{split}
\norm{M}_{2, T}^2 & =  \ME\sup\limits_{0 \le t \le T}M^2(t)\\
& \le 4\ME M^2(T)\\
& = 4\int_0^T\left[\si(\mathcal Z(s)) - \si(\mathcal Z'(s))\right]^2\,\md s \\
& \le
4C^2(\si)T\cdot\ME\sup\limits_{0 \le t \le T}{(\mathcal Z(t) - \mathcal Z'(t))}^2\\
& = 4C^2(\si)T\norm{\mathcal Z - \mathcal Z'}^2_{2, T}.
\end{split}
\end{align}
Combining~\eqref{eq:norm-diff},~\eqref{eq:term-I},~\eqref{eq:term-M}, we complete the proof of~\eqref{eq:Lipschitz-general}.
\end{proof}

For small enough $T$, the constant $C_T$ from~\eqref{eq:C-T} is strictly less than $1$. Assume this is the case until the end of the proof. Then for every $\mathcal X \in \mathcal E_T$, the mapping $\mathcal R_{\mathcal X}$ is contractive. Therefore, it has a unique fixed point, which can be obtained by successive approximations:
$$
\mathcal Y(\mathcal X) = \lim\limits_{n \to \infty}\mathcal R_{\mathcal X}^n(\mathcal Z).
$$
In particular, the equation~\eqref{eq:fixed-point-equation} has a unique solution, which is obtained by successive approximations:
$$
Y_{\eps} = \lim\limits_{n \to \infty}\mathcal R^n_{\mathcal X_{\eps}}(\mathcal Z).
$$
We can take $\mathcal Z = 0$ as initial condition, or any other element in $\mathcal E_T$. Applying the mappings in Lemma~\ref{lemma:Lipschitz-general} once again, we have:
$$
\norm{\mathcal R^2_{\mathcal X}(\mathcal Z) - \mathcal R^2_{\mathcal X'}(\mathcal Z')}_{2, T}  \le C_T^2\norm{\mathcal Z - \mathcal Z'} + (1+ C_T)\norm{\mathcal X - \mathcal X'}.
$$
By induction over $n = 1, 2, \ldots$ we get:
\begin{align}
\label{eq:iteration-Lipschitz}
\begin{split}
\norm{\mathcal R^n_{\mathcal X}(\mathcal Z) - \mathcal R^n_{\mathcal X'}(\mathcal Z')}_{2, T} \le C_T^n\norm{\mathcal Z - \mathcal Z'}_{2, T} + \left(1 + C_T + \ldots + C_T^{n-1}\right)\norm{\mathcal X - \mathcal X'}_{2, T}.
\end{split}
\end{align}
Let $n \to \infty$ in~\eqref{eq:iteration-Lipschitz}. If $C_T < 1$, then
\begin{equation}
\label{eq:Lipschitz-solution}
\norm{\mathcal Y(\mathcal X) - \mathcal Y(\mathcal X')}_{2, T} \le \frac1{1 - C_T}\norm{\mathcal X - \mathcal X'}_{2, T}.
\end{equation}
Letting $\mathcal X = \mathcal X_0$ and $\mathcal X' = \mathcal X_{\eps}$ in~\eqref{eq:Lipschitz-solution}, and using~\eqref{eq:eps-0}, we complete the proof of Lemma~\ref{lemma:fundamental}.

\section{Concluding remarks}
We showed that the convergence of ruin probabilities in a rather broad class of risk processes is achieved exponentially fast. This rate is easy to compute (at least in the examples considered in Section \ref{sec:ComputationConvergenceRate}), and happened to be sharp when the premium rate and its variability are independent from the current wealth of the insurance company. A natural question relies on the practical implication of having access to the value of the rate of exponential convergence; in particular, whether this leads to an numerical approximation of the finite time ruin probability. This issue has been discussed in Asmussen \cite{As84}, the answer was negative. Another direction is to relax the condition upon the tail of the claim size. It is of practical interest to let the claim size distribution be heavy tailed. An extension of the early work of Asmussen and Teugels~\cite{AsTe96} could be envisaged. For example, in the work of Tang \cite{Ta05}, a compound Poisson risk model under constant interest force with sub-exponentially distributed claim size is considered. When comparing the asymptotics provided by Tang \cite[(2.5), (3.2)]{Ta05}, it seems that exponential convergence holds for large initial reserves. Yet another direction for future research might be to relax the Lipschitz property of the drift.

\section*{Acknowledgements} Pierre-Olivier Goffard was partially funded by a Center of Actuarial Excellence Education Grant given to the University of California, Santa Barbara, from the Society of Actuaries. Andrey Sarantsev was supported in part by the NSF grant DMS 1409434 (with Jean-Pierre Fouque as a Principal Investigator) during this work.

\appendix
\section{Proof of Lemma~\ref{lemma:tech}}\label{appendix:ProofLemma1} We combine Assumption~\ref{asmp:finite-exp} with~\eqref{eq:classic-mean} to conclude this. Indeed, from Assumption~\ref{asmp:finite-exp} it follows that
\begin{equation}
\label{eq:int-1-infty}
\int_1^{\infty}x\,\mu(\md x) < \infty,
\end{equation}
and from~\eqref{eq:classic-mean} we conclude that
\begin{equation}
\label{eq:int-0-1}
\int_0^1x\,\mu(\md x) < \infty.
\end{equation}
Condition~\eqref{eq:Mean} then immediately follows from~\eqref{eq:int-0-1} and~\eqref{eq:int-1-infty}.
\section{Proof of Lemma~\ref{lem:ExistenceRProcess}}\label{appendix:ProofLemma3} Using the notation similar to the proof of Lemma~\ref{lemma:fundamental}, we need to find the fixed point of the mapping $\mathcal R_L$. But from Lemma~\ref{lemma:Lipschitz-general}, we get that this mapping $\mathcal R_L$ is $C_T$-Lipschitz with $C_T$ taken from~\eqref{eq:C-T}. For small enough $T$, we have $C_{T} < 1$, and therefore the fixed point exists and is unique by the classic theorem. Thus we can prove strong existence and pathwise uniqueness on the time interval $[0, T]$, and then on $[T, 2T]$, $[2T, 3T]$, etc. The form of the generator then follows from straightforward application of It\^o's formula.
\section{Proof of Lemma~\ref{lem:ExistenceXProcess}}\label{appendix:ProofLemma2} Similar to the proof of Lemma~\ref{lemma:ExistenceRProcess}, but without reflection; therefore we can take an identity map instead of $\mathcal R_{\mathcal X}$, which is of course $1$-Lipschitz. The rest of the proof works verbatim.

\section{Proof of Lemma~\ref{lemma:stoch-ordered-original}} 
\label{appendix:ProofStochOrder}
Consider two copies $X_1$ and $X_2$ of this process, starting from $X_1(0) = x_1$ and $X_2(0) = x_2$, where $x_1 > x_2 \ge 0$. Let us couple them: that is, we create their copies on a common probability space. using the same driving Brownian motion $W$ and L\'evy process $L$. We can do this by Lemma~\ref{lem:ExistenceXProcess}. Next, we aim to prove that $X_1(t) \ge X_2(t)$ for all $t \ge 0$ simultaneously, with probability $1$. This would automatically imply that $\MP(X_1(t) \ge c) \ge \MP(X_2(t) \ge c)$ for all $t, c \ge 0$, which is the property (b) in Theorem~\ref{thm:stoch-order}.

\smallskip

Assume there exists a $t > 0$ such that $X_1(t) < X_2(t)$. Let $\tau := \inf\{t \ge 0\mid X_1(t) < X_2(t)\}$. By right-continuity of $X_1$ and $X_2$, we must have $X_1(\tau) \le X_2(\tau)$. But we cannot have $X_1(\tau) = X_2(\tau)$, because then by strong Markov property we would have $X_1(t) = X_2(t)$ for all $t \ge \tau$ (recall that $\tau$ is a stopping time). Therefore,
\begin{equation}
\label{eq:contradiction-1}
X_1(\tau) < X_2(\tau),\ \mbox{but}\ X_1(\tau-) \ge X_2(\tau-).
\end{equation}
Thus, $\tau$ is a jump time for both $X_1$ and $X_2$, that is, for the L\'evy process $L$. The displacement during the jump must be the same for $X_1$ and $X_2$:
\begin{equation}
\label{eq:contradiction-2}
X_1(\tau) - X_1(\tau-) = -\left[L(\tau) - L(\tau-)\right] = X_2(\tau) - X_2(\tau-).
\end{equation}
The contradiction between~\eqref{eq:contradiction-1} and~\eqref{eq:contradiction-2} completes the proof of Lemma~\ref{lemma:stoch-ordered-original}.

\end{document}